\documentclass[a4paper,11pt]{article}
\usepackage{mathrsfs, amsmath, amsxtra, amssymb, latexsym, amscd, amsthm}
\usepackage{indentfirst}
\usepackage[mathcal]{eucal}
\usepackage{graphicx,color}

\theoremstyle{plain}
{\theoremstyle{plain}\newtheorem{theorem}{Theorem}[section]}
\newtheorem{cor}{Corollary}[section]
{\theoremstyle{plain}\newtheorem{proposition}{Proposition}[section]}
{\theoremstyle{plain}\newtheorem{lemma}{Lemma}[section]}
{\theoremstyle{plain}}
{\theoremstyle{plain}\newtheorem{rem}{Remark}[section]}
\newtheorem{lem}{Lemma}[section]

\numberwithin{equation}{section}

\begin{document}
\title
{Fisher information bounds and applications to SDEs with small noise}
\author{Nguyen Tien Dung\thanks{Department of Mathematics, VNU University of Science, Vietnam National University, Hanoi, 334 Nguyen
Trai, Thanh Xuan, Hanoi, 084 Vietnam.}\,\,\footnote{Corresponding author. Email: dung@hus.edu.vn}
\and Nguyen Thu Hang\thanks{Department of Mathematics, Hanoi University of Mining and Geology, 18 Pho Vien, Bac Tu Liem, Hanoi, 084 Vietnam.}
}

\date{\today}          
\maketitle
\begin{abstract}In this paper, we first establish general bounds on the Fisher information distance to the class of normal distributions of Malliavin differentiable random variables. We then study the rate of Fisher information convergence in the central limit theorem for the solution of small noise stochastic differential equations and its additive functionals. We also show that the convergence rate is of optimal order.
\end{abstract}

\section{Introduction}
Given a random variable $F$ with an absolutely continuous density $p_F$,
the Fisher information of $F$ (or its distribution) is defined by
$$
I(F)= \int_{-\infty}^{+\infty} \frac{p_F'(x)^2} {p_F(x)}dx=E[\rho_F^2(F)],
$$
where $p_F'$ denotes a Radon-Nikodym derivative of $p_F$ and $\rho_F:=p_F'/p_F$ is the score function. Furthermore, the Fisher information distance of $F$ to the normal distribution $N\sim \mathcal{N}(\mu,\sigma^2)$ is defined by
$$I(F\|N):=E\left[\left(\rho_F(F)+\frac{F-\mu}{\sigma^2}\right)^2\right].$$
If the derivative $p_F'$  does not exist then the Fisher information distance is defined to be infinite. We refer the reader to the monograph \cite{Johnson2004} for various properties if $I(F)$ and $I(F\|N).$ Here we note that $I(F\|N)$ is not a distance. However, it dominates many important distances in statistical applications such as the relative entropy (or Kullback-Leibler distance) and the supremum distance between densities. It also dominates the traditional distances such as Kolmogorov distance, total variation distance and Wasserstein distance, etc. Thus $I(F\|N)$  provides a very strong measure of convergence to normality when studying the central limit theorem.

The study of Fisher information convergence has a long history beginning in 1959 with the results of Linnik \cite{Linnik1959}. However, most of the existing results are devoted to the sums of independent random variables. For such sums, the quantitative estimates for the rate of convergence have been well studied, see \cite{Bobkov2014,B-J,Johnson2004,Johnson2020}. Recently, Nourdin \& Nualart \cite{Nourdin1} used the techniques of Malliavin calculus to obtain quantitative Fisher information bounds for the multiple Wiener-It\^o integrals. This is a remarkable contribution to the literature because, for the first time, the quantitative estimates for $I(F\|N)$ were obtained for the random variables $F$ not to be a sum of independent random variables.

In the present paper, our first purpose is to extend the method developed in \cite{Nourdin1} to a general class of Malliavin differentiable random variables. We provide two explicit estimates for the Fisher information distance in  Theorems \ref{lm2} and \ref{9hj}. Our second purpose is to investigate the rate of Fisher information convergence in the central limit theorem for the solution of stochastic differential equations (SDEs) with small noise:
\begin{align}\label{eq3}
X_{\varepsilon,t}=X_0+\int_0^tb(s,X_{\varepsilon,s})ds+\varepsilon\int_0^t\sigma(s,X_{\varepsilon,s})dB_s,\,\,\,t\in [0,T],
\end{align}
where the initial condition $X_0$ is a real number, $b,\sigma:[0,T]\times \mathbb{R}\to \mathbb{R}$ are deterministic functions, $(B_t)_{t\in [0, T]}$ is a standard Brownian motion and $\varepsilon\in (0,1)$ is a small parameter.
Let us consider the ordinary differential equation
\begin{align}\label{eq4}
x_{t}=X_0+\int_0^tb(s,x_s)ds,\mbox{  }t\in [0,T].
\end{align}
It is well known that, in theory of stochastic differential equations with small noise, one of the fundamental problems is to study the convergence of $X_{\varepsilon,t}$ to $x_t$ as $\varepsilon$ tends zero. The convergence can be described via large deviation principle, central limit theorem and moderate deviation principle, etc. (see two monographs \cite{Freidlin2012,Kutoyants1994} for more details). In particular, the central limit theorem results have been discussed by various authors. Under suitable assumptions, for example, it follows from \cite{Spiliopoulos2014,Suo2018} that, as $\varepsilon\to0,$
$$\tilde{X}_{\varepsilon,t}:=\frac{X_{\varepsilon,t}-x_t}{\varepsilon}\,\,\,\text{converges in distribution to}\,\,\,N_t$$
for every $t\in [0,T],$ where $N_t$ is a centered normal random variable with appropriate variance. Naturally, one may wonder whether the convergence also holds in the Fisher information distance (i.e. convergence holds in a stronger sense). Our Theorem \ref{klf9} below not only gives  an affirmative answer, but also provides a convergence rate of order $O(\varepsilon^2).$ In addition, in Theorem \ref{k7zlf}, we show that the rate $O(\varepsilon^2)$ is optimal as $\varepsilon\to0.$ In Theorem \ref{theorem2}, as a further illustration, we also obtain the optimal rate $O(\varepsilon^2)$ of Fisher information convergence for the additive functional of the form
\begin{align}\label{eq1}
	Y_{\varepsilon,t} =  \int_0^tf(s, X_{\varepsilon,s})ds,\mbox{  }t\in [0,T].
\end{align}
The rest of the paper is organized as follows. In Section \ref{sec1}, we recall some fundamental concepts of Malliavin calculus. Section \ref{sec2} contains the abstract results of this paper, two upper bounds on the Fisher information distance are given in Theorems \ref{lm2} and \ref{9hj}. Section \ref{sec2e} is devoted to the study of Fisher information convergence for the solution of (\ref{eq3}) and its additive functional (\ref{eq1}). The main results of this section is formulated and proved in Theorems \ref{klf9}, \ref{k7zlf} and \ref{theorem2}. An useful estimate for the negative moment of Volterra functionals is given in Appendix.
\section{Preliminaries}\label{sec1}
As we have said in the introduction, this paper is concerned with the Fisher information distance via the techniques of Malliavin calculus. For the reader's convenience, let us recall some elements of Malliavin calculus (for more details see \cite{nualartm2}). We suppose that $(B_t)_{t\in [0, T]}$ is defined on a complete probability space $(\Omega ,\mathcal{F},\mathbb{F},P)$, where $\mathbb{F}= (\mathcal{F}_t)_{t \in [0,T]}$ is a natural filtration generated by the Brownian motion $B$. For $h \in L^2[0, T]$, we denote by $B(h)$ the Wiener integral $$B(h)= \int_0^Th(t)dB_t.$$
Let $\mathcal{S}$ denote a dense subset of $L^2(\Omega ,\mathcal{F}, P)$ that consists of smooth random variables of the form
\begin{align}
	F = f(B(h_1),B(h_2),..., B(h_n) ),\label{iik}
\end{align}
where $n \in \mathbb{N}, f \in C_b^\infty(\mathbb{R}^n)$ and $ h_1, h_2, ..., h_n \in L^2[0,T]$. If $F$ has the form (\ref{iik}), we define its Malliavin derivative as the process $DF:= {D_tF, t\in [0,T]}$ given by
$$D_tF =\sum_{k=1}^{n}\frac{\partial f}{\partial x_k}(B(h_1),B(h_2),..., B(h_n))h_k(t).$$
More generally, for each $k\ge 1,$ we can define the iterated derivative operator on a cylindrical random variable by setting
$$ D_{t_1,...,t_k}^{k}F=D_{t_1}...D_{t_k}F. $$
For any $1 \le p,k< \infty$, we denote by $\mathbb{D}^{k,p}$ the closure of $\mathcal{S}$ with respect to the norm
$$||F||_{k,p}^p:= E|F|^p + E\left[\bigg(\int_0^T|D_uF|^2du\bigg)^{\frac{p}{2}}\right]+...+E\left[\bigg(\int_0^T...\int_0^T|D^{k}_{t_1,...,t_k}F|^2dt_1...dt_k\bigg)^{\frac{p}{2}}\right].$$
A random variable $F$ is said to be Malliavin differentiable if it belongs to $\mathbb{D}^{1,2}$. For any $F\in \mathbb{D}^{1,2},$ the Clark-Ocone formula says that
$$ F-E[F]=\int_0^{T}E[D_sF|\mathcal{F}_s]dB_s. $$
In particular, we have
$${\rm Var}(F)=\int_0^{T}E(E[D_sF|\mathcal{F}_s])^2ds=\int_0^{T}E\left[D_sFE[D_sF|\mathcal{F}_s]\right]ds. $$
An important operator in the Malliavin's calculus theory is the divergence operator $\delta$. It is the adjoint of derivative operator $D.$ The domain of $\delta$ is the set of all functions $u\in L^2(\Omega \times [0,T])$ such that
$$ E|\langle DF,u\rangle _{L^2[0,T]}|\le C(u)\|F\|_{L^2(\Omega)}, $$
where $C(u)$ is some positive constant depending on $u.$ In particular, if $u\in Dom\delta ,$ then $\delta(u)$ is characterized by following duality relationships
\begin{align}
\delta(uF)&=F\delta(u)-\left\langle DF,u\right\rangle_{L^2[0,T]}\label{ct2}\\
E[\left\langle DF,u\right\rangle_{L^2[0,T]}]&=E[F\delta (u)] \mbox{ for any } F\in \mathbb{D}^{1,2}.\label{ct3}
\end{align}
It is known that any random variable $F$ in $L^2(\Omega, \mathcal{F},P)$ can be expanded into an orthogonal sum of its Wiener chaos:
$$F=\sum\limits_{n=0}^\infty J_nF,$$
where $J_0=E[F]$ and $J_n$ denotes the projection onto the $n$th Wiener chaos. From this chaos expansion one may define the Ornstein-Uhlenbeck operator $L$ by
$$LF=\sum\limits_{n=0}^\infty-n J_nF.$$
The domain of $L$ is
\begin{equation*}
Dom L=\{F\in L^2(\Omega ):\sum_{n=1}^{\infty }n^{2}E\left|
J_{n}F\right|^{2}<\infty \}=\mathbb{D}^{2,2}\text{.}
\end{equation*}
Moreover, a random variable $F$ belongs to $Dom L$ if and only if $F\in Dom\delta D$ (i.e. $F\in \mathbb{D}^{1,2}$ and $DF\in Dom\delta$), and in this case: $\delta DF=-LF.$ We also define the operator $L^{-1}$ as follows: for every $F\in L^2(\Omega)$ with zero mean, we set
$$L^{-1}F=\sum\limits_{n=1}^\infty-\frac{1}{n} J_nF.$$
Note that, for any $F \in L^2(\Omega )$ with zero mean, we have that $L^{-1} F \in  \mathrm{Dom}L$,
and
\begin{equation*}
LL^{-1} F = F.
\end{equation*}

\section{General Fisher information bounds}\label{sec2}
We first construct the representation formula for the score function.
\begin{lemma}\label{uhuy}
Let $F\in \mathbb{D}^{1,2}$ and $u:\Omega \rightarrow L^2[0,T]$, and
suppose that $\left\langle DF,u\right\rangle_{L^2[0,T]}\neq 0$ a.s. and
$\frac{u}{\left\langle DF,u\right\rangle_{L^2[0,T]} }$ belongs to the domain of $\delta.$ Then the law of $F$ has an absolutely continuous density and its score function $\rho_F$ is given by
\begin{equation}\label{kofkl}
\rho_F(x):=p_F'(x)/p_F(x)=-E\left[\delta \left(
\frac{u}{\left\langle DF,u\right\rangle_{L^2[0,T]}}\right)\big| F=x \right],\,\,\,x\in \mathrm{supp}\,p_F.
\end{equation}
\end{lemma}
\begin{proof}This lemma is not new. For the sake of completeness, we shall give a proof here. According to Exercise 2.1.3 in \cite{nualartm2}, the law of $F$ has a continuous density
given by
\begin{equation}\label{Fmla2}
p_{F}\left( x\right) =E\left[ \mathbf{1}_{\left\{ F>x\right\} }\delta \left(
\frac{u}{\left\langle DF,u\right\rangle_{L^2[0,T]}}\right) \right],\,\,\, x\in \mathrm{supp}\,p_F.
\end{equation}
Note that the proof of (\ref{Fmla2}) is similar to that of Proposition 2.1.1 in \cite{nualartm2}. Since $F\in\mathbb{D}^{1,2},$ this implies that $\mathrm{supp}\,\rho_F$ is  a closed interval of $\mathbb{R}$ (see Proposition 2.1.7 in \cite{nualartm2}): $\mathrm{supp}\,\rho_F=[\alpha,\beta]$ with $-\infty\leq \alpha<\beta\leq \infty.$ It follows from (\ref{Fmla2}) that
\begin{align*}
p_{F}\left( x\right)&=E\left[ \mathbf{1}_{\left\{ F>x\right\} }E\left[\delta \left(
\frac{u}{\left\langle DF,u\right\rangle_{L^2[0,T]} }\right)\big| F\right]\right]\\
&=\int_x^\beta E\left[\delta \left(
\frac{u}{\left\langle DF,u\right\rangle_{L^2[0,T]} }\right)\big| F=y\right]p_{F}(y)dy.
\end{align*}
So $p_F$ is absolutely continuous and the representation (\ref{kofkl})  is verified. The proof of the lemma is complete.
\end{proof}
We now are ready to establish the Fisher information bounds by using suitable choices of the function $u.$
\begin{theorem}\label{lm2}Let $F\in \mathbb{D}^{2,4}$ and $N$ be a normal random variable with mean $\mu$ and variance $\sigma^2.$ Define
$$\Theta:=\left\langle DF,u \right\rangle_{L^2[0,T]},$$
where $u_t:=E[D_tF|\mathcal{F}_t],\,t\in[0,T].$  Assume that $\Theta\neq 0$ a.s. Then, we have
\begin{align}\label{ufaq}
I(F\|N)\leq c\left(\frac{1}{\sigma^4}(E[F]-\mu)^2+A_F|{\rm Var}(F)-\sigma^2|^2+C_{F}\left(E\|D\Theta\|_{L^2[0,T]}^4\right)^{1/2}\right),
\end{align}
where $c$ is an absolute constant and $A_F,C_{F}$ are positive constants given by
$$A_F:=\frac{1}{\sigma^4}\left(E\|u\|_{L^2[0,T]}^8E|\Theta|^{-8}\right)^{1/4},\,\,\,C_{F}:=A_F+\left(E\|u\|_{L^2[0,T]}^8E|\Theta|^{-16}\right)^{1/4}.$$
\end{theorem}
\begin{proof} For simplicity, we write $\langle.,. \rangle$ instead of $\langle.,. \rangle_{L^2[0,T]}$ and $\|.\|$ instead of $\|.\|_{L^2[0,T]}.$
 In Lemma \ref{uhuy}, we  use $u_t:=E[D_tF|\mathcal{F}_t],\,t\in[0,T].$ Note that, by the Clark-Ocone formula, we have
$$F-E[F]=\int_0^TE[D_tF|\mathcal{F}_t]dB_t=\delta(u).$$
Hence, by the integration-by-part formula (\ref{ct2}), we get the following representation for the score function
\begin{align*}
\rho_F(F)=-E\left[\frac{F-E[F]}{\Theta}+\frac{\left\langle D\Theta,u \right\rangle}{\Theta^2}\Big|F\right].
\end{align*}
As a consequence, we obtain
\begin{align*}
I(F\|N)&=E\left(\rho_F(F)+\frac{F-\mu}{\sigma^2}\right)^2\\
&=E\left(E\left[-\frac{F-E[F]}{\Theta}-\frac{\left\langle D\Theta,u \right\rangle}{\Theta^2}+\frac{F-\mu}{\sigma^2}\Big|F\right]\right)^2\\
&\leq E\left(-\frac{(F-E[F])(\sigma^2-\Theta)}{\sigma^2\Theta}-\frac{\left\langle D\Theta,u \right\rangle}{\Theta^2}+\frac{E[F]-\mu}{\sigma^2}\right)^2.
\end{align*}
 Using the Cauchy-Schwarz and H\"older inequalities we deduce
\begin{align}
&I(F\|N)\leq 3E\left[\frac{(F-E[F])^2(\Theta-\sigma^2)^2}{\sigma^4\Theta^2}\right]+3E\left[\frac{\left\langle D\Theta,u \right\rangle^2}{\Theta^4}\right]+\frac{3}{\sigma^4}(E[F]-\mu)^2\notag\\
&\leq \frac{3}{\sigma^4}\left(E|F-E[F]|^8E|\Theta|^{-8}\right)^{1/4}\left(E|\Theta-\sigma^2|^{4}\right)^{1/2}+3E\left[\frac{\|D\Theta\|^2\|u\|^2 }{|\Theta|^4}\right]+\frac{3}{\sigma^4}(E[F]-\mu)^2\notag\\
&\leq \frac{3}{\sigma^4}\left(E|F-E[F]|^8E|\Theta|^{-8}\right)^{1/4}\left(E|\Theta-\sigma^2|^{4}\right)^{1/2}
+3\left(E\|D\Theta\|^{4}\right)^{1/2}\left(E\|u\|^8E|\Theta|^{-16}\right)^{1/4}\notag\\
&+\frac{3}{\sigma^4}(E[F]-\mu)^2.\label{i1}
\end{align}
We note that $E[\Theta]={\rm Var}(F).$ Then, by the inequality  (3.19) in \cite{Nourdin2009} we have
\begin{align}
E|\Theta-\sigma^2|^{4}&\leq 8|{\rm Var}(F)-\sigma^2|^{4}+8E|\Theta-{\rm Var}(F)|^{4}\notag\\
&\leq 8|{\rm Var}(F)-\sigma^2|^{4}+72E\left[\|D\Theta\|^4\right].\label{i2}
\end{align}
On the other hand, by the Burkholder-Davis-Gundy inequality, there exists $C>0$ such that
\begin{align}
E|F-E[F]|^8&=E\left[\bigg(\int_0^TE[D_tF|\mathcal{F}_t]dB_t\bigg)^8\right]\notag\\
&\leq CE\left[\bigg(\int_0^T|E[D_tF|\mathcal{F}_t]|^2dt\bigg)^4\right]\notag\\
&=CE\|u\|^8.\label{i3}
\end{align}
So we can get the desired estimate (\ref{ufaq}) by inserting (\ref{i2}) and (\ref{i3}) into (\ref{i1}). This completes the proof of the theorem.
\end{proof}
\begin{theorem}\label{9hj}Let $F\in \mathbb{D}^{2,4}$ and $N$ be a normal random variable with mean $\mu$ and variance $\sigma^2.$ Define
$$\Gamma:=\left\langle DF,-DL^{-1}F \right\rangle_{L^2[0,T]}.$$
Assume that $\Gamma\neq 0$ a.s. Then, we have
\begin{align}\label{rde}
I(F\|N)\leq c\left(\frac{1}{\sigma^4}(E[F]-\mu)^2+A_F|{\rm Var}(F)-\sigma^2|^2+C_{F}\left(E\|D\Gamma\|_{L^2[0,T]}^4\right)^{1/2}\right),
\end{align}
where $c$ is an absolute constant and $A_F,C_{F}$ are positive constants given by
$$A_F:=\frac{1}{\sigma^4}\left(E\|DF\|_{L^2[0,T]}^8E|\Gamma|^{-8}\right)^{1/4},\,\,\,C_{F}:=A_F+\left(E\|DF\|_{L^2[0,T]}^8E|\Gamma|^{-16}\right)^{1/4}.$$
\end{theorem}
\begin{proof} Consider the stochastic process $u_t:=-D_tL^{-1}F,\,t\in[0,T].$  Note that
\begin{equation}\label{ct3xz}
F - E[F]=LL^{-1} F =-\delta DL^{-1} F=\delta(u).
\end{equation}
Hence, with the exact proof of (\ref{i1}), we obtain the following.
\begin{align}
I(F\|N)
&\leq \frac{3}{\sigma^4}\left(E|F-E[F]|^8E|\Gamma|^{-8}\right)^{1/4}\left(E|\Gamma-\sigma^2|^{4}\right)^{1/2}
\notag\\
&+3\left(E\|D\Gamma\|_{L^2[0,T]}^{4}\right)^{1/2}\left(E\|u\|_{L^2[0,T]}^8E|\Gamma|^{-16}\right)^{1/4}+\frac{3}{\sigma^4}(E[F]-\mu)^2.\label{it1}
\end{align}
By (\ref{ct3}) and (\ref{ct3xz}) we have ${\rm Var}(F)=E[F\delta(u)]=E[\Gamma].$ Hence, by the inequality  (3.19) in \cite{Nourdin2009}, we obtain
\begin{align}
E|\Gamma-\sigma^2|^{4}&\leq 8|{\rm Var}(F)-\sigma^2|^{4}+8E|\Gamma-{\rm Var}(F)|^{4}\notag\\
&\leq 8|{\rm Var}(F)-\sigma^2|^{4}+72E\left[\|D\Gamma\|_{L^2[0,T]}^4\right].\label{i2z}
\end{align}
We also have
\begin{equation}\label{i2z1}
E|F-E[F]|^8\leq 7^4E\|DF\|_{L^2[0,T]}^8.
\end{equation}
On the other hand, by the inequality  (3.17) in \cite{Nourdin2009}, we have
\begin{equation}\label{i2z2}
E\|u\|_{L^2[0,T]}^8\leq E\|DF\|_{L^2[0,T]}^8.
\end{equation}
Inserting (\ref{i2z}), (\ref{i2z1}) and (\ref{i2z2}) into (\ref{it1}) yields the desired bound (\ref{rde}). So the proof of the theorem is complete.
\end{proof}

\begin{rem} (i) We have implicitly assumed that the bounds (\ref{ufaq}) and (\ref{rde}) both involve finite quantities, as otherwise there is nothing to prove.

\noindent(ii) In general, the random variables $\Theta$ and $\Gamma$ are different from each other. However, they satisfy the relationship: $E[\Theta|F]=E[\Gamma|F]\,a.s.$ (see Proposition 2.3 in \cite{Dung2015}).

\end{rem}
\begin{rem}(i)  Theorem \ref{lm2} is of interest for the readers who are not used to working with the Ornstein-Uhlenbeck operator. On the other hand, comparing with Theorem \ref{lm2}, the advantage of Theorem \ref{9hj} lies in the fact that it can be extended to a more general setting: Suppose that $\mathfrak{H}$ is a real separable Hilbert space with scalar product denoted by $\langle.,.\rangle_\mathfrak{H}.$ We denote by $W = \{W(h) : h \in \mathfrak{H}\}$ an isonormal Gaussian process defined in a complete probability space $(\Omega,\mathcal{F},P),$ $\mathcal{F}$ is the $\sigma$-field generated by $W.$ Now Malliavin derivative operator is with respect to $W.$ Then, we have
\begin{equation}\label{yfh}
I(F\|N)\leq c\left(\frac{1}{\sigma^4}(E[F]-\mu)^2+A_F|{\rm Var}(F)-\sigma^2|^2+C_{F}\left(E\|D\Gamma\|_{\mathfrak{H}}^4\right)^{1/2}\right),
\end{equation}
where $\Gamma:=\left\langle DF,-DL^{-1}F \right\rangle_{\mathfrak{H}}$ and
$$A_F:=\frac{1}{\sigma^4}\left(E\|DF\|_{\mathfrak{H}}^8E|\Gamma|^{-8}\right)^{1/4},C_{F}:=A_F+\left(E\|DF\|_{\mathfrak{H}}^8E|\Gamma|^{-16}\right)^{1/4}.$$
The bound (\ref{yfh}) thus provides us a potential tool to study the Fisher information distance for stochastic differential equations driven by fractional Brownian motion, or stochastic partial differential equations.


\noindent(ii) Let $F=I_q(f)$ be a multiple Wiener-It\^o integral of order $q\geq 2.$ We have $-DL^{-1}F=\frac{1}{q}DF$ and hence, $\Gamma=\frac{1}{q}\|DF\|_{\mathfrak{H}}^2.$ We now use the moment estimates for $D^2F\otimes_1 DF$ provided in \cite{Nourdin1} and we obtain
$$\left(E\|D\Gamma\|_{\mathfrak{H}}^4\right)^{1/2}\leq c(E|F|^4-3),$$
where $c$ is a positive constant. So our bound  (\ref{yfh}) recovers the fourth moment bound established in \cite{Nourdin1} for the multiple Wiener-It\^o integrals.
\end{rem}
\section{Optimal Fisher information bounds for small noise SDEs}\label{sec2e}
In this Section, we apply Theorem \ref{lm2} to investigate the rate of Fisher information convergence for the solution to the equation (\ref{eq3}) and its additive functional (\ref{eq1}). Although  theory of stochastic differential equations with small noise is very rich, to the best of our knowledge, the results of this section are new. Given a function $h(t,x),$ we use the notations
$$h'(t,x)=\frac{\partial h(t,x)}{\partial{x}} \quad \text{and}\quad h''(t,x)=\frac{\partial ^2h(t,x)}{\partial x^2}.$$
We make the use of the following assumptions:
 \begin{itemize}
 	\item[$(A_{1})$] The coefficients $b, \sigma :[0,T]\times \mathbb{R}\to \mathbb{R}$ are measurable functions having linear growth, that is, there exists $L>0$ such that
 $$ |b(t,x)| +|\sigma(t,x)|\le L(1+|x|)\,\,\, \forall x\in \mathbb{R},t\in[0,T].$$
 	\item [$(A_2)$]$\sigma(t,x)$ and $b(t,x)$ are twice continuously differentiable in $x$ with the derivatives bounded by $L.$
 	\item [$(A_3)$] $f(t,x)$ is twice differentiable in $x,$ $f(t,x)$ together with its derivatives have polynomial growth and $\|f'\|_0:=\inf\limits_{(t,x)\in [0,T]\times \mathbb{R}}f'(t,x)>0.$
 \end{itemize}
The main results of this section are stated in the following theorems.
\begin{theorem}\label{klf9} Let $(X_{\varepsilon,t})_{t\in[0,T]}$ and $(x_t)_{t\in[0,T]}$ be the solutions to the equations (\ref{eq3}) and (\ref{eq4}), respectively. Define
$$\tilde{X}_{\varepsilon,t}:=\frac{X_{\varepsilon,t}-x_t}{\varepsilon},\,\,\,\beta_t^2:=\int_0^t \sigma^2(r,x_r)\exp \left(2\int_r^t b'(u,x_u)du  \right)dr,\,\,t\in[0,T].$$
Suppose the assumptions $(A_1)$-$(A_2)$ and that, for some $p_0>16,$
\begin{equation}\label{gy3}
E\left[\bigg(\int_0^t\sigma^2(r,X_{\varepsilon,r})dr\bigg)^{-p_0}\right]<\infty\,\,\,\forall\,\varepsilon\in (0,1),t\in(0,T].
\end{equation}
Then, for all $\varepsilon\in (0,1)$ and $t\in(0,T],$ we have
\begin{multline}
I(\tilde{X}_{\varepsilon,t}\|N_t)\leq C\bigg(\frac{t^4}{\beta_t^4}+\frac{t^4}{\beta_t^4}\bigg(E\bigg|\int_0^t\sigma^2(r,X_{\varepsilon,r})dr\bigg|^{-p_0}\bigg)^{\frac{2}{p_0}}\\+
t^4\bigg(E\bigg|\int_0^t\sigma^2(r,X_{\varepsilon,r})dr\bigg|^{-p_0}\bigg)^{\frac{4}{p_0}}\bigg)\varepsilon^2,\label{ukf}
\end{multline}
where $N_t$ denotes a normal random variable with mean zero and variance $\beta_t^2$ and $C$ is a positive constant not depending on $t$ and $\varepsilon.$
\end{theorem}
We can write down $C$ in an explicit form ($\log C$ is a polynomial in $T$ and $L$), but it is not our goal here. The non-degeneracy condition (\ref{gy3}) make Theorem \ref{klf9} not easy to use in practical applications. Hence, it is necessary to obtain the sufficient conditions which are easy to check. We have the following.
\begin{cor}\label{9lg}Suppose the assumptions $(A_1)$-$(A_2).$ We assume, in addition, that $\sigma(0,X_0)\neq 0$ and $|\sigma(t,x)-\sigma(s,x)|\leq L|t-s|^{\delta_1}$ for all $x\in \mathbb{R}$ and $s,t\in[0,T],$ where $L,\delta_1$ are positive real numbers. Then, we have
\begin{equation}\label{wu5f}
I(\tilde{X}_{\varepsilon,t}\|N_t)\leq C\left(\frac{t^4}{\beta_t^4}+\frac{t^2}{\beta_t^4}+1\right)\varepsilon^2\,\,\,\forall\,\varepsilon\in (0,1),t\in(0,T],
\end{equation}
where $C$ is a positive constant not depending on $t$ and $\varepsilon.$
\end{cor}
\begin{proof}It is easy to see that $E|X_{\varepsilon,t}-X_0|^p\leq Ct^{\frac{p}{2}}$ for all $\varepsilon\in (0,1),t\in[0,T],$ where $C$ is a positive constant not depending on $t$ and $\varepsilon.$ On the other hand, the function $\sigma$ satisfies
$$|\sigma(t,x)-\sigma(s,y)|\leq L(|t-s|^{\delta_1}+|x-y|)\,\,\,\forall x,y\in \mathbb{R},s,t\in[0,T].$$
Hence, we can apply Lemma A (given in Appendix) to $Y_t=X_{\varepsilon,t},$ $h(t,x)=\sigma(t,x)$ and $k(t,s)=1$ and we obtain
$$
E\left[\bigg(\int_0^t\sigma^2(r,X_{\varepsilon,r})dr\bigg)^{-p_0}\right]\leq Ct^{-p_0},\,\,t\in(0,T]
$$
for all $p_0>0.$ So the bound (\ref{wu5f}) follows directly from (\ref{ukf}).
\end{proof}
\begin{rem}Generally, the Berry-Esseen bounds for the rate of convergence are more informative in practice. As an application of Corollary \ref{9lg}, we obtain the following
$$\sup\limits_{x\in \mathbb{R}}\left|P(\tilde{X}_{\varepsilon,t}\leq x)-P(N_t\leq x)\right|\leq \sqrt{I(\tilde{X}_{\varepsilon,t}\|N_t)}\leq C\left(\frac{t^2}{\beta_t^2}+\frac{t}{\beta_t^2}+1\right)\varepsilon\,\,\,\forall\,\varepsilon\in (0,1),t\in(0,T].$$
\end{rem}
The bound (\ref{gy3}) provides us the convergence rate of order $O(\varepsilon^2)$ as $\varepsilon\to 0.$ One may wonder whether this rate is optimal. Interestingly, the answer is affirmative as in the next theorem.
\begin{theorem}\label{k7zlf}Suppose the assumptions $(A_1)$-$(A_2).$ Then, for each $t\in(0,T],$ we have
\begin{equation}\label{mm1}
\lim\limits_{\varepsilon\to 0}\frac{1}{\varepsilon^2}I(\tilde{X}_{\varepsilon,t}\|N_t)\geq \frac{1}{4\beta_t^4} \left(E\left|E\left[\delta\left(V_tDU_t\right)|U_t\right]\right|\right)^2,
\end{equation}
where $\beta_t^2$ is as in Theorem \ref{klf9}, $(U_t)_{t\in[0,T]}$ and $(V_t)_{t\in[0,T]}$ are stochastic processes defined by
\begin{equation}\label{ydlq1}
U_t=\int_0^tb'(s,x_s)U_sds+\int_0^t\sigma(s,x_s)dB_s,\,\,0 \leq t\leq T,
\end{equation}
\begin{equation}\label{ydlq2}
V_t=\int_0^t\left(\frac{1}{2}b''(s,x_s)U_s^2+b'(s,x_s)V_s\right)ds+\int_0^t\sigma'(s,x_s)U_sdB_s,\,\,0 \leq t\leq T.
\end{equation}
\end{theorem}
\begin{rem}The stochastic differential equations (\ref{ydlq1}) and (\ref{ydlq2})  can be solved explicitly. We have
$$U_t=\int_0^t\sigma(s,x_s)e^{\int_s^tb'(u,x_u)du}dB_s,\,\,0 \leq t\leq T$$
and
$$V_t=\int_0^t\frac{1}{2}b''(s,x_s)U_s^2e^{\int_s^tb'(u,x_u)du}ds+\int_0^t\sigma'(s,x_s)U_se^{\int_s^tb'(u,x_u)du}dB_s,\,\,0 \leq t\leq T.$$
Furthermore, the random variables $U_t$ and $V_t$ are Malliavin differentiable and their derivatives are given by
$$D_rU_t=\sigma(r,x_r)e^{\int_r^tb'(u,x_u)du},\,\,0 \leq r\leq t\leq T,$$
\begin{multline*}
D_rV_t=\sigma'(r,x_r)U_re^{\int_r^tb'(u,x_u)du}\\
+\int_r^tb''(s,x_s)U_sD_rU_se^{\int_s^tb'(u,x_u)du}ds+
\int_r^t\sigma'(s,x_s)D_rU_se^{\int_s^tb'(u,x_u)du}dB_s,\,\,0 \leq r\leq t\leq T.
\end{multline*}
It is also easy to see that $U_t,V_t\in \mathbb{D}^{k,p}$ for all $k\geq 1,p\geq 2.$
\end{rem}
In the next theorem, for the additive functional of solutions, we also obtain the convergence rate of optimal order $O(\varepsilon^2).$
\begin{theorem}\label{theorem2}Consider the stochastic process $(Y_{\varepsilon,t})_{t\in[0,T]}$  defined by (\ref{eq1}). Define $y_t:=\int_0^{t}f(s,x_s)ds$ and
	$$ \tilde{Y}_{\varepsilon,t}:=\frac{Y_{\varepsilon,t}-y_t}{\varepsilon},\mbox{  }\gamma_t^{2}:=\int_0^t \left(\int_r^tf'(s,x_s)\sigma (r,x_r)e ^{\int_r^s b'(u,x_u)du}ds\right)^2dr,\,\,0 \leq t\leq T.$$
Suppose the assumptions $(A_1)$-$(A_3)$  and that, for some $p_0>16,$
\begin{equation}\label{g4y3}
E\left[\bigg(\int_0^t(t-r)^2\sigma^2(r,X_{\varepsilon,r})dr\bigg)^{-p_0}\right]<\infty\,\,\,\forall\,\varepsilon\in (0,1),t\in(0,T].
\end{equation}
Then, for all $\varepsilon\in (0,1)$ and $t\in(0,T],$ we have
\begin{multline}
I(\tilde{Y}_{\varepsilon,t}\|Z_t)
\leq C\bigg(\frac{t^4}{\gamma_t^4}+\frac{t^{10}}{\gamma_t^4}\bigg(E\bigg|\int_0^t(t-r)^2\sigma^2(r,X_{\varepsilon,r})dr\bigg|^{-p_0}\bigg)^{\frac{2}{p_0}}\\+
t^{10}\bigg(E\bigg|\int_0^t(t-r)^2\sigma^2(r,X_{\varepsilon,r})dr\bigg|^{-p_0}\bigg)^{\frac{4}{p_0}}\bigg)\varepsilon^2,\label{jjdz}
\end{multline}
where $Z_t$ denotes a normal random variable with mean zero and variance $\gamma_t^2$ and $C$ is a positive constant not depending on $t$ and $\varepsilon.$
\end{theorem}
\begin{cor}Suppose the assumptions $(A_1)$-$(A_3).$ We assume, in addition, that $\sigma(0,X_0)\neq 0$ and $|\sigma(t,x)-\sigma(s,x)|\leq L|t-s|^{\delta_1}$ for all $x\in \mathbb{R}$ and $s,t\in[0,T],$ where $L,\delta_1$ are positive real numbers. Then, we have
\begin{equation}\label{r3f}
I(\tilde{Y}_{\varepsilon,t}\|Z_t)\leq C\left(\frac{t^4}{\gamma_t^4}+\frac{1}{t^2}\right)\varepsilon^2\,\,\,\forall\,\varepsilon\in (0,1),t\in(0,T],
\end{equation}
where $C$ is a positive constant not depending on $t$ and $\varepsilon.$
\end{cor}
\begin{proof}For every $t_0\in (0,T],$ we have
$$\int_0^{t_0}(t_0-r)^2\sigma^2(r,X_{\varepsilon,r})dr=t_0^3\int_0^1(1-r)^2\sigma^2(t_0r,X_{\varepsilon,t_0r})dr.$$
We consider the stochastic process $Y_t:=X_{\varepsilon,t_0t},0\leq t\leq 1.$ Then, $Y_0=X_0$ is deterministic and $E|Y_t-X_0|^p=E|X_{\varepsilon,t_0t}-X_0|^p\leq Ct^{\frac{p}{2}}$ for all $\varepsilon\in (0,1),t\in[0,1],$ where $C$ is a positive constant not depending on $t_0,t$ and $\varepsilon.$ In addition, the functions $h(t,x)=\sigma(t_0t,x)$ and $k(t,s)=(t-s)^2$ satisfy
\begin{align*}
&|h(t,x)-h(s,y)|\leq L(|t-s|^{\delta_1}+|x-y|)\,\,\,\forall x,y\in \mathbb{R},s,t\in[0,1],\\
&|k(t,s)-k(t,0)|\leq L|t-s|\,\,\,\forall s,t\in[0,1]
\end{align*}
for some positive constant $L$ not depending on $t_0.$ Thus, for all $p_0>0,$ we can use Lemma A to get
$$E\left[\bigg(\int_0^1(1-r)^2\sigma^2(t_0r,X_{\varepsilon,t_0r})dr\bigg)^{-p_0}\right]\leq C<\infty,$$
where $C$ is a positive constant not depending on $t_0$ and $\varepsilon.$ Consequently,
$$E\left[\bigg(\int_0^{t_0}(t_0-r)^2\sigma^2(r,X_{\varepsilon,r})dr\bigg)^{-p_0}\right]\leq Ct_0^{-3p_0}\,\,\,\forall\,t_0\in (0,T],$$
and hence, the bound (\ref{r3f}) follows directly from (\ref{jjdz}).
\end{proof}
\begin{rem} (i) In the assumption $(A_3),$ the condition $\|f'\|_0:=\inf\limits_{(t,x)}f'(t,x)>0$ can be replaced by $\|f'\|_0:=\sup\limits_{(t,x)\in  [0,T]\times \mathbb{R}}f'(t,x)<0.$

\noindent(ii) Consider the stochastic processes
$$\overline{U}_t =  \int_0^tf'(s, x_s)U_sds,\,\,0\leq t\leq T,$$
$$\overline{V}_t =  \frac{1}{2}\int_0^t(f''(s, x_s)U_s^2+2f'(s,x_s)V_s)ds,\,\,0\leq t\leq T,$$
where $(U_t)_{t\in[0,T]}$ and $(V_t)_{t\in[0,T]}$ are as in Theorem \ref{k7zlf}. The reader can verify that
\begin{equation*}
\lim\limits_{\varepsilon\to 0}\frac{1}{\varepsilon^2}I(\tilde{Y}_{\varepsilon,t}\|Z_t)\geq \frac{1}{4\gamma_t^4} \left(E\left|E\left[\delta\left(\overline{V}_tD\overline{U}_t\right)|\overline{U}_t\right]\right|\right)^2,\,\,t\in(0,T].
\end{equation*}
The proof is similar to that of (\ref{mm1}). So we omit it.
\end{rem}

\subsection{Estimates for Malliavin derivatives}
 Hereafter,  we denote by $C$ a generic constant which may vary at each appearance. Let us collect some fundamental results about the Malliavin differentiability of solutions to the equation (\ref{eq3}).
\begin{proposition}\label{pr1}
Suppose the assumptions $(A_1)$ and $(A_2).$ Then, the equation (\ref{eq3}) has a unique solution $(X_{\varepsilon,t})_{t\in[0,T]}$ satisfying, for each $p\geq 2,$
\begin{align}\label{s4hd}
\sup\limits_{t\in[0,T]} E|X_{\varepsilon,t}|^p\le C,\,\,\,\forall\,\varepsilon\in (0,1),
\end{align}
where $C$ is a positive constant not depending on $\varepsilon.$ Moreover, for each $t\in[0,T],$ the random variable $X_{\varepsilon,t}$ is twice Malliavin differentiable and the derivatives satisfy the following linear equations, for all $0\le r,\theta \le t\le T,$
\begin{align}
D_rX_{\varepsilon,t}= \varepsilon \sigma(r,X_{\varepsilon,r})+\int_r^tb'(s,X_{\varepsilon,s})D_rX_{\varepsilon,s}ds+\varepsilon\int_r^t\sigma'(s,X_{\varepsilon,s})
D_rX_{\varepsilon,s}dB_s\label{dhc1x}
\end{align}
and
\begin{align}
&	D_{\theta} D_rX_{\varepsilon,t}=\varepsilon\sigma'(r,X_{\varepsilon,r})D_{\theta}X_{\varepsilon,r}+\int_{r\vee\theta}^t\left[b''(s,X_{\varepsilon,s})
D_{\theta}X_{\varepsilon,s}D_rX_{\varepsilon,s}+b'(s,X_{\varepsilon,s})D_{\theta}D_rX_{\varepsilon,s}\right]ds\notag\\
&+ \varepsilon\sigma'(\theta, X_{\varepsilon,\theta})D_rX_{\varepsilon,\theta}+\varepsilon\int_{r\vee\theta}^t\left[\sigma''(s,X_{\varepsilon,s})D_{\theta}X_{\varepsilon,s}
D_rX_{\varepsilon,s}+\sigma'(s,X_{\varepsilon,s})D_{\theta}D_rX_{\varepsilon,s}\right]dB_s.\label{dhc1x8}
\end{align}
\end{proposition}
\begin{proof} See Theorems 2.2.1 and 2.2.2 in \cite{nualartm2}.
\end{proof}
\begin{proposition}\label{pr5}
Suppose the assumptions $(A_1)$ and $(A_2).$   Then, for each $p \ge 2,$ we have
	 \begin{align}\label{ct1}
	 \sup\limits_{0\le r \le t\le T}E|D_rX_{\varepsilon,t}|^p\le C\varepsilon^p\,\,\,\forall\,\varepsilon\in (0,1),
	\end{align}
and
\begin{align}\label{ct1a}
\sup\limits_{0\le r,\theta \le t\le T}E| D_{\theta}D_rX_{\varepsilon,t}|^p\le C\varepsilon^{2p}\,\,\,\forall\,\varepsilon\in (0,1),
\end{align}
	where $C$ is a positive constant not depending on $\varepsilon.$
\end{proposition}
\begin{proof}The proof is similar to that of Theorems 2.2.1 and 2.2.2 in \cite{nualartm2}. For each $p\ge2,$ by the fundamental inequality $(|a_1|+|a_2|+|a_3|)^p\leq 3^{p-1}(|a_1|^p+|a_2|^p+|a_3|^p),$ we obtain from (\ref{dhc1x}) that
\begin{multline*}
|D_rX_{\varepsilon,t}|^p\le 3^{p-1}\left|\varepsilon\sigma(r,X_{\varepsilon,r})\right|^p+3^{p-1}\left|\int_r^tb'(s,X_{\varepsilon,s})D_rX_{\varepsilon,s}ds\right|^p\\
			+3^{p-1}\left|\varepsilon\int_r^t\sigma'(s,X_{\varepsilon,s})D_rX_{\varepsilon,s}dB_s\right|^p
\end{multline*}
for all $0\leq r\leq t\leq T.$ From the linear growth property of $\sigma$ and the estimate (\ref{s4hd}), we have
\begin{equation}\label{iif}
\sup\limits_{0\leq r\leq T} E\left|\sigma(r,X_{\varepsilon,r})\right|^p\le C\,\,\,\forall\,\varepsilon\in (0,1),
\end{equation}
where $C$ is a positive constant not depending on $\varepsilon.$ Furthermore, by using the H\"{o}lder and Burkholder-Davis-Gundy inequalities, we deduce
		\begin{align*}
			E\left|\int_r^tb'(s,X_{\varepsilon,s})D_rX_{\varepsilon,s}ds\right|^p&\le L^p(t-r)^{p-1}\int_r^tE|D_rX_{\varepsilon,s}|^pds\\
			& \le L^pT^{p-1} \int_r^tE|D_rX_{\varepsilon,s}|^pds,\,\,0\leq r\leq t\leq T
		\end{align*}
and, for some $C_p>0,$
\begin{align*}			E\left|\varepsilon\int_r^t\sigma'(s,X_{\varepsilon,s})D_rX_{\varepsilon,s}dB_s\right|^p&\le C_p\varepsilon^pE\left(\int_r^t\left|\sigma'(s,X_{\varepsilon,s})
D_rX_{\varepsilon,s}\right|^2ds\right)^{\frac{p}{2}}\\
&\le C_p(t-r)^{\frac{p}{2}-1}\varepsilon^pL^p\int_r^tE|D_rX_{\varepsilon,s}|^pds\\
&\le C_pT^{\frac{p}{2}-1}L^p \varepsilon^p\int_r^tE|D_rX_{\varepsilon,s}|^pds,\,\,0\leq r\leq t\leq T.
		\end{align*}
We therefore obtain, for all $\varepsilon\in (0,1),$
\begin{align*}
	E|D_rX_{\varepsilon,t}|^p\le C\varepsilon^p+C\int_r^tE|D_rX_{\varepsilon,s}|^pds,\,\,0\leq r\leq t\leq T,
\end{align*}
where $C$ is a positive constant not depending on $r,t$ and $\varepsilon.$ Using Gronwall's lemma, we get
$$ E|D_rX_{\varepsilon,t}|^p\le C\varepsilon^pe^{C(t-r)}\le C\varepsilon^p \mbox{  }\forall \mbox{  } 0\le r\le t\le T. $$
This finishes the proof of (\ref{ct1}). The proof of (\ref{ct1a}) can be done similarly. Indeed, we obtain from the equation (\ref{dhc1x8}) that
\begin{align*}
	E| D_{\theta}D_rX_{\varepsilon,t}|^p& \le 4^{p-1}E\left|\varepsilon\sigma'(r,X_{\varepsilon,r})D_{\theta}X_{\varepsilon,r}\right|^{p} +4^{p-1}E\left| \varepsilon\sigma'(\theta, X_{\theta})D_rX_{\varepsilon,\theta}\right|^{p}\\
	&+4^{p-1} E\left|\int_{r\vee	\theta}^t\left[b''(s,X_{\varepsilon,s})D_{\theta}X_{\varepsilon,s}D_rX_{\varepsilon,s}+b'(s,X_{\varepsilon,s})D_{\theta}D_rX_{\varepsilon,s}\right]ds\right|^{p}\\ &+4^{p-1}E\left|\varepsilon\int_{r\vee\theta}^t\left[\sigma''(s,X_{\varepsilon,s})D_{\theta}X_{\varepsilon,s}D_rX_{\varepsilon,s}
+\sigma'(s,X_{\varepsilon,s})D_{\theta}D_rX_{\varepsilon,s}\right]dB_s\right|^{p}
\end{align*}
 for all $ 0\le \theta,r \le t\le T.$ By using the estimate (\ref{ct1}) and the boundedness of $b', b'', \sigma',\sigma''$ we obtain
\begin{align*}
	E\left|\varepsilon\sigma'(r,X_{\varepsilon,r})D_{\theta}X_{\varepsilon,r}\right|^{p}+E\left| \varepsilon\sigma'(\theta, X_{\theta})D_rX_{\varepsilon,\theta}\right|^{p}\le C\varepsilon^{2p},
\end{align*}
\begin{align*} E\bigg|\int_{r\vee\theta}^t\left[b''(s,X_{\varepsilon,s})D_{\theta}X_{\varepsilon,s}D_rX_{\varepsilon,s}+b'(s,X_{\varepsilon,s})D_{\theta}D_rX_{\varepsilon,s}\right]ds\bigg|^{p}\\
	\le C\varepsilon^{2p}+C\int_{r\vee\theta}^t E\left|D_{\theta} D_rX_{\varepsilon,s}\right|^pds,
\end{align*}
and
\begin{align*}
&E\left|\varepsilon\int_{r\vee\theta}^t\left[\sigma''(s,X_{\varepsilon,s})D_{\theta}X_{\varepsilon,s}D_rX_{\varepsilon,s}+\sigma'(s,X_{\varepsilon,s})D_{\theta}D_rX_{\varepsilon,s}\right]dB_s\right|^{p}
\\
&\qquad\qquad \qquad=\varepsilon^{p}E\left|\int_{r\vee\theta}^t\left[\sigma''(s,X_{\varepsilon,s})D_{\theta}X_{\varepsilon,s}D_rX_{\varepsilon,s}+\sigma'(s,X_{\varepsilon,s})D_{\theta}D_rX_{\varepsilon,s}\right]^{2}ds\right|^{\frac{p}{2}}\qquad\qquad\qquad
\\
&\qquad\qquad\qquad\le C\varepsilon^{3p}+C\varepsilon^p\int_{r\vee\theta}^t E\left|D_{\theta} D_rX_{\varepsilon,s}\right|^pds.
\end{align*}
As a consequence, for all $\varepsilon\in (0,1),$
$$E| D_{\theta}D_rX_{\varepsilon,t}|^p\le  C\varepsilon^{2p}+C\int_{r\vee\theta}^t E\left|D_{\theta} D_rX_{\varepsilon,s}\right|^pds,\,\,0\le \theta,r\le t\le T$$
and we obtain (\ref{ct1a}) by using Gronwall's lemma again.

The proof the proposition is complete.
\end{proof}
\subsection{Proof of Theorem \ref{klf9}}
The proof of Theorem \ref{klf9} will be given at the end of this subsection. In order to be able to apply Theorem \ref{lm2}, we need the following technical results.
\begin{lem}\label{pxro1}Suppose the assumptions $(A_1)$-$(A_2).$ Then, for each $p\geq 2,$ we have
\begin{equation}\label{oo1}
E|X_{\varepsilon,t}-x_t|^p\le Ct^{\frac{p}{2}} \varepsilon^p\,\,\,\forall\,\varepsilon\in (0,1),t\in[0,T],
\end{equation}
where $C$ is a positive constant not depending on $t$ and $\varepsilon.$ Moreover, if $h(t,x)$ is a differentiable function in $x$ and  its derivative has polynomial growth, we also have
\begin{equation}\label{ozo1}
E|h(t,X_{\varepsilon,t})-h(t,x_t)|^p\le Ct^{\frac{p}{2}} \varepsilon^p\,\,\,\forall\,\varepsilon\in (0,1),t\in[0,T].
\end{equation}
\end{lem}
\begin{proof} We have
$$X_{\varepsilon,t}-x_{t}=\int_0^t(b(s,X_{\varepsilon,s})-b(s,x_s))ds+\varepsilon\int_0^t\sigma(s,X_{\varepsilon,s})dB_s,\,\,t\in[0,T].$$
Hence, by H\"{o}lder and Burkholder-Davis-Gundy inequalities, we deduce
\begin{align*}
E|X_{\varepsilon,t}-x_t|^p&=E\left|\int_0^t(b(s,X_{\varepsilon,s})-b(s,x_s))ds+\varepsilon\int_0^t\sigma(s,X_{\varepsilon,s})dB_s\right| ^p
\\
&\le  2^{p-1}E\left|\int_0^t(b(s,X_{\varepsilon,s})-b(s,x_s))ds\right|^p+2^{p-1}\varepsilon^pE\left|\int_0^t\sigma(s,X_{\varepsilon,s})dB_s\right|^p
\\
&\le Ct^{p-1}\int_0^tE|b(s,X_{\varepsilon,s})-b(s,x_s)|^pds+Ct^{\frac{p}{2}-1}\varepsilon^p \int_0^tE|\sigma(s,X_{\varepsilon,s})|^pds,\,\,t\in[0,T],
\end{align*}
where $C$ is a positive constant not depending on $t$ and $\varepsilon.$ So by the boundedness of $b'$ and the estimate (\ref{iif}) we obtain
\begin{align*}
E|X_{\varepsilon,t}-x_t|^p&\le C\int_0^tE|X_{\varepsilon,s}-x_s|^pds +Ct^{\frac{p}{2}}\varepsilon^p,\,\,t\in[0,T],
\end{align*}
which, together with Gronwall's lemma, yields
$$ E|X_{\varepsilon,t}-x_t|^p\le Ct^{\frac{p}{2}}\varepsilon^p e^{Ct}\leq Ct^{\frac{p}{2}} \varepsilon^p,\,\,t\in[0,T].$$
It remains to prove (\ref{ozo1}). For each $t\in[0,T],$ using the Taylor's expansion, we have
$$h(t,X_{\varepsilon,t})-h(t,x_t)=h'(t,x_t+\eta _t(X_{\varepsilon,t}-x_t))(X_{\varepsilon,t}-x_t),$$
where $\eta _t$ is a random variable lying between 0 and 1. By the polynomial growth property of $h'$ and the estimate (\ref{s4hd}), we have $\sup\limits_{t\in[0,T]}E|h'(t,x_t+\eta _t(X_{\varepsilon,t}-x_t))|^p\leq C$ for all $\varepsilon\in (0,1),$ where $C$ is a positive constant not depending on $\varepsilon.$ We now use the Cauchy-Schwarz inequality and the estimate (\ref{oo1}) to get
$$E|h(t,X_{\varepsilon,t})-h(t,x_t)|^p\leq \sqrt{E|h'(t,x_t+\eta _t(X_{\varepsilon,t}-x_t))|^{2p}E|X_{\varepsilon,t}-x_t|^{2p}}\leq Ct^{\frac{p}{2}} \varepsilon^p.$$
This finishes the proof of the proposition.
\end{proof}
\begin{proposition}\label{pro1}Suppose the assumptions $(A_1)$-$(A_2).$ Let $(\tilde{X}_{\varepsilon,t})_{t\in[0,T]}$ be as in Theorem \ref{klf9}.
Then, we have
\begin{align}
|E[\tilde{X}_{\varepsilon,t}]|&\leq Ct^2\varepsilon\,\,\,\forall\,\varepsilon\in (0,1),t\in[0,T],\label{oo2}\\
|{\rm Var}(\tilde{X}_{\varepsilon,t})-\beta_t^2|&\leq Ct^{3/2}\varepsilon\,\,\,\forall\,\varepsilon\in (0,1),t\in[0,T],\label{oo3}
\end{align}
where $C$ is a positive constant not depending on $t$ and $\varepsilon.$
\end{proposition}
\begin{proof} We first verify the estimate (\ref{oo2}). By using the Taylor's expansion, we obtain
\begin{align}
&\tilde{X}_{\varepsilon,t}=\frac{1}{\varepsilon}\int_0^t(b(s,X_{\varepsilon,s})-b(s,x_s))ds+\int_0^t\sigma(s,X_{\varepsilon,s})dB_s\notag\\
&=\int_0^tb'(s,x_s)\tilde{X}_{\varepsilon,s}ds
+\frac{1}{2\varepsilon}\int_0^tb''(s,x_s+\theta_s(X_{\varepsilon,s}-x_s))(X_{\varepsilon,s}-x_s)^2ds+\int_0^t\sigma(s,X_{\varepsilon,s})dB_s,\label{ydlq}
\end{align}
where, for each $0\leq s\leq t,$ $\theta_s$ is a random variable lying between $0$ and $1.$ Taking the expectation of $\tilde{X}_{\varepsilon,t}$ gives us
\begin{align*}
E[\tilde{X}_{\varepsilon,t}]
&=\int_0^tb'(s,x_s)E[\tilde{X}_{\varepsilon,s}]ds
+\frac{1}{2\varepsilon}\int_0^tE[b''(s,x_s+\theta_s(X_{\varepsilon,s}-x_s))(X_{\varepsilon,s}-x_s)^2]ds
\end{align*}
This is a linear differential equation and its solution is given by
$$E[\tilde{X}_{\varepsilon,t}]=\frac{1}{2\varepsilon}\int_0^te^{\int_s^tb'(u,x_u)du}
E[b''(s,x_s+\theta_s(X_{\varepsilon,s}-x_s))(X_{\varepsilon,s}-x_s)^2]ds.$$
Consequently, we can use the boundedness of the derivatives $b',b''$ and the estimate (\ref{oo1}) to get
\begin{align*}
|E[\tilde{X}_{\varepsilon,t}]|&\leq\frac{Le^{LT}}{2\varepsilon}\int_0^tE|X_{\varepsilon,s}-x_s|^2ds\\
&\leq Ct^2\varepsilon,\,\,t\in[0,T].
\end{align*}
So the estimate (\ref{oo2}) holds. It remains to prove the estimate (\ref{oo3}). Solving the equation (\ref{dhc1x}) we obtain
\begin{align}
D_rX_{\varepsilon,t} &=\varepsilon \sigma(r,X_{\varepsilon,r})\exp\left(\int_r^t\left(b'(u,X_{\varepsilon,u}) -\frac{1}{2}\varepsilon^2\sigma'^2(u,X_{\varepsilon,u})\right)du  +\varepsilon\int_r^t\sigma'(u,X_{\varepsilon,u})dB_u \right)\notag\\
&=\varepsilon \sigma(r,X_{\varepsilon,r})\exp\left(\int_r^tb'(u,X_{\varepsilon,u})du\right)Z_{r,t},\,\,\,0\leq r\leq t\leq T.\label{jknm}
\end{align}
where $Z_{r,t}$ is given by
\begin{equation}\label{alne}
Z_{r,t}:=\exp\left(\varepsilon\int_r^t\sigma'(u,X_{\varepsilon,u})dB_u-\frac{1}{2}\varepsilon^2\int_r^t\sigma'^2(u,X_{\varepsilon,u})du\right),\,\,0\leq r\leq t\leq T.
\end{equation}
Note that, by the It\^o differential formula, $Z_{r,t}$ satisfies
\begin{equation}\label{axlne}
Z_{r,t}=1+\int_r^t\varepsilon\sigma'(s,X_{\varepsilon,s})Z_{r,s}dB_s,\,\,0\leq r\leq t\leq T.
\end{equation}
So $E[Z_{r,t}|\mathcal{F}_r]=1.$ Furthermore, for each $p
\geq 2,$ it is easy to see that
\begin{equation}\label{iif2}
\sup\limits_{0\leq r\leq t\leq T}E|Z_{r,t}|^p\leq C
\end{equation}
for some $C>0$ not depending on $\varepsilon\in (0,1).$ By the Clark-Ocone formula, the It\^o isometry and (\ref{jknm}), we have
\begin{align*}
 {\rm Var}(\tilde{X}_{\varepsilon,t})&=E\left[\int_0^t (E[D_r\tilde{X}_{\varepsilon,t}|\mathcal{F}_r])^2dr\right]=\frac{1}{\varepsilon^2}E\left[\int_0^t (E[D_rX_{\varepsilon,t}|\mathcal{F}_r])^2dr\right]\\
 &=E\left[\int_0^t \sigma^2(r,X_{\varepsilon,r})\left(E\left[e^{\int_r^t b'(u,X_{\varepsilon,u})du  }Z_{r,t}\Big|\mathcal{F}_r\right]\right)^2dr\right],
\end{align*}
and hence,
\begin{align}
&{\rm Var}(\tilde{X}_{\varepsilon,t})-\beta_t^2\notag\\
&=E\left[\int_0^t \sigma^2(r,X_{\varepsilon,r})\left(E\left[e^{\int_r^t b'(u,X_{\varepsilon,u})du}Z_{r,t}\Big|\mathcal{F}_r\right]\right)^2dr\right]-\int_0^t \sigma^2(r,x_r)e^{2\int_r^t b'(u,x_u)du}dr\notag\\
&=E\left[\int_0^t \left(\sigma^2(r,X_{\varepsilon,r})-\sigma^2(r,x_r)\right)\left(E\left[e^{\int_r^t b'(u,X_{\varepsilon,u})du  }Z_{r,t}\Big|\mathcal{F}_r\right]\right)^2dr\right]\notag\\
&\qquad\qquad+E\left[\int_0^t \sigma^2(r,x_r)\left(\left(E\left[e^{\int_r^t b'(u,X_{\varepsilon,u})du}Z_{r,t}\Big|\mathcal{F}_r\right]\right)^2-e^{2\int_r^t b'(u,x_u)du}\right)dr\right].\label{oms}
\end{align}
Consequently, since $b'$ is bounded by $L,$ $E[Z_{r,t}|\mathcal{F}_r]=1$ and $\sup\limits_{0\leq r\leq T}\sigma^2(r,x_r)<\infty,$ we can infer from (\ref{oms}) that
\begin{align}
&|{\rm Var}(\tilde{X}_{\varepsilon,t})-\beta_t^2|\notag\\
&\leq C\int_0^t E|\sigma^2(r,X_{\varepsilon,r})-\sigma^2(r,x_r)|dr+C\int_0^t E\left|E\left[e^{\int_r^t b'(u,X_{\varepsilon,u})du  }Z_{r,t}\Big|\mathcal{F}_r\right]-e^{\int_r^t b'(u,x_u)du}\right|dr\notag\\
&\leq C\int_0^t E|\sigma^2(r,X_{\varepsilon,r})-\sigma^2(r,x_r)|dr+C\int_0^t E\left|\bigg(e^{\int_r^t b'(u,X_{\varepsilon,u})du  }-e^{\int_r^t b'(u,x_u)du}\bigg)Z_{r,t}\right|dr\notag\\
&\leq C\int_0^t E|\sigma^2(r,X_{\varepsilon,r})-\sigma^2(r,x_r)|dr+C\int_0^t \int_r^t E|(b'(u,X_{\varepsilon,u})-b'(u,x_u))Z_{r,t}|dudr\notag\\
&\leq C\int_0^t E|\sigma^2(r,X_{\varepsilon,r})-\sigma^2(r,x_r)|dr+C\int_0^t \int_r^t \sqrt{E|b'(u,X_{\varepsilon,u})-b'(u,x_u)|^2E|Z_{r,t}|^2}dudr,\label{nvn}
\end{align}
where $C$ is a positive constant not depending on $t$ and $\varepsilon.$ Hence, in view of the estimates (\ref{ozo1}) and (\ref{iif2}), we get
\begin{equation}\label{nvn1}
|{\rm Var}(\tilde{X}_{\varepsilon,t})-\beta_t^2| \le Ct^{\frac{3}{2}}\varepsilon + Ct^{\frac{5}{2}}\varepsilon \le Ct^{\frac{3}{2}}\varepsilon .
\end{equation}
So the estimate (\ref{oo3}) is proved. This completes the proof of the proposition.
\end{proof}

\begin{proposition}\label{momenam1}Let $(\tilde{X}_{\varepsilon,t})_{t\in[0,T]}$ be as in Theorem \ref{klf9}. Define
$$\Theta_{\tilde{X}_{\varepsilon,t}}:=\int_0^tD_r\tilde{X}_{\varepsilon,t}E[D_r\tilde{X}_{\varepsilon,t}|\mathcal{F}_r]dr,\,\,t\in [0,T].$$
Then, under the assumption of Theorem \ref{klf9}, we have
\begin{align}\label{mmam1}
 E|\Theta_{\tilde{X}_{\varepsilon,t}}|^{-p}\le C\left(E\bigg|\int_0^t\sigma^2(r,X_{\varepsilon,r})dr\bigg|^{-p_0}\right)^{\frac{p}{p_0}}\,\,\,\forall\,\varepsilon\in (0,1),t \in (0,T],
\end{align}
where $0<p<p_0$ and $C>0$ is a positive constant not depending on $t$ and $\varepsilon.$
\end{proposition}
\begin{proof}
Since $D_r\tilde{X}_{\varepsilon,t}=\frac{D_rX_{\varepsilon,t} }{\varepsilon},$ it follows from (\ref{jknm}) that
\begin{align*}
\Theta_{\tilde{X}_{\varepsilon,t}}&=\frac{1}{\varepsilon^2}\int_0^tD_rX_{\varepsilon,t}E[D_rX_{\varepsilon,t}|\mathcal{F}_r]dr\\
&=\int_0^t\sigma^2(r,X_{\varepsilon,r})e^{\int_r^tb'(u,X_{\varepsilon,u})du}Z_{t,r}E\left[e^{\int_r^tb'(u,X_{\varepsilon,u})du}Z_{t,r}\big|\mathcal{F}_r\right]dr,
\end{align*}
where $Z_{t,r}$ is defined by (\ref{alne}). Since $b',\sigma'$ are bounded by $L,$ we deduce $e^{\int_r^tb'(u,X_{\varepsilon,u})du}Z_{t,r}\geq e^{-\frac{3LT}{2}}e^{\varepsilon\int_r^t\sigma'(u,X_{\varepsilon,u})dB_u}$ and $E\left[e^{\int_r^tb'(u,X_{\varepsilon,u})du}Z_{t,r}\big|\mathcal{F}_r\right]\geq e^{-LT}E\left[Z_{t,r}\big|\mathcal{F}_r\right]=e^{-LT}.$ Then, we obtain
\begin{align*}
\Theta_{\tilde{X}_{\varepsilon,t}}&\geq e^{-\frac{5LT}{2}}\int_0^t\sigma^2(r,X_{\varepsilon,r})e^{\varepsilon\int_r^t\sigma'(u,X_{\varepsilon,u})dB_u }dr\\
&\geq e^{-\frac{5LT}{2}}e^{M_t-\max\limits_{0\le t\le T}M_t}\int_0^t\sigma^2(r,X_{\varepsilon,r})dr\\
&\geq e^{-\frac{5LT}{2}}e^{-2\max\limits_{0\le t\le T}M_t}\int_0^t\sigma^2(r,X_{\varepsilon,r})dr,
\end{align*}
where $M_t:=\varepsilon\int_0^t\sigma'(u,X_{\varepsilon,u})dB_u,0\le t\le T.$ We observe that $M_t$ is a martingale with the bounded quadratic variation. Indeed, $\langle M \rangle _t=\varepsilon^2\int_0^t\sigma'^2(s,X_{\varepsilon,s})ds\le L^2T$ for all $\varepsilon\in(0,1)$ and $t\in[0,T].$ By Dubin and Schwarz's theorem, there exists a one dimensional Brownian motion $(m_t)_{t\ge 0}$ such that $M_t=m_{\langle M \rangle _t}$. Then we arrive at the following
\begin{align*}
	|\Theta_{\tilde{X}_{\varepsilon,t}}|\ge e^{-\frac{5LT}{2}}e^{-2\max\limits_{0\le t\le L^2T}m_t}\int_0^t\sigma^2(r,X_{\varepsilon,r})dr\,\,\,\forall\,\varepsilon\in (0,1),t \in (0,T].
\end{align*}
Note that, by  Fernique's theorem, we always have $E\left[e^{4q\max\limits_{0\leq t\leq L^2T}m_t}\right]<\infty$ for all $q>0.$ Hence, for $0<p<p_0$, we use H\"older's inequality to deduce the following
\begin{align*}
	E|\Theta_{\tilde{X}_{\varepsilon,t}}|^{-p}&\leq e^{\frac{5pLT}{2}}E\left[e^{2p\max\limits_{0\le t\le T}m_t}\bigg(\int_0^t\sigma^2(r,X_{\varepsilon,r})dr\bigg)^{-p}\right]\\
&\leq C\left(E\bigg|\int_0^t\sigma^2(r,X_{\varepsilon,r})dr\bigg|^{-p_0}\right)^{\frac{p}{p_0}}\,\,\,\forall\,\varepsilon\in (0,1),t \in (0,T],
\end{align*}
where $C>0$ is a positive constant not depending on $t$ and $\varepsilon.$ The proof of the proposition is complete.
\end{proof}

\begin{proposition}\label{pr3.7}Let $(\Theta_{\tilde{X}_{\varepsilon,t}})_{t\in[0,T]}$ be as in Proposition \ref{momenam1}. Suppose the assumptions $(A_1)$ and $(A_2).$ Then, we have
 \begin{align}\label{Dgamma1}
E\|D\Theta_{\tilde{X}_{\varepsilon,t}}\|^4_{L^2[0, T]}\le C \varepsilon^{4}t^6\,\,\,\forall\,\varepsilon\in (0,1),t \in [0,T],
\end{align}
where $C$ is a positive constant not depending on $t$ and $\varepsilon.$
\end{proposition}
\begin{proof}
The Malliavin derivative of $\Theta_{\tilde{X}_{\varepsilon,t}}$ can be computed as follows
\begin{align}\label{dhg}
	D_{\theta}\Theta_{\tilde{X}_{\varepsilon,t}} =\frac{1}{\varepsilon^2}\int_{0}^t\left(D_{\theta} D_rX_{\varepsilon,t} E[D_rX_{\varepsilon,t}|\mathcal{F}_r] + D_rX_{\varepsilon,t}E[D_{\theta}D_rX_{\varepsilon,t}|\mathcal{F}_r]\right)dr.
\end{align}
Hence, we get
\begin{align*}
	\| D\Theta_{\tilde{X}_{\varepsilon,t}}\|^4_{L^2[0, T]}=\frac{1}{\varepsilon^8}\left( \int_0^t \left(\int_{0}^t\left(D_{\theta} D_rX_{\varepsilon,t} E[D_rX_{\varepsilon,t}|\mathcal{F}_r] + D_rX_{\varepsilon,t}E[D_{\theta}D_rX_{\varepsilon,t}|\mathcal{F}_r]\right)dr\right)^2d\theta\right) ^{2}.
\end{align*}
 Using the H\"{o}lder inequality gives us
 \begin{align*}
	E\| D\Theta_{\tilde{X}_{\varepsilon,t}}&\|^4_{L^2[0, T]}\leq \frac{t}{\varepsilon^8}\int_0^t E\bigg|\int_{0}^t\left(D_{\theta} D_rX_{\varepsilon,t} E[D_rX_{\varepsilon,t}|\mathcal{F}_r] + D_rX_{\varepsilon,t}E[D_{\theta}D_rX_{\varepsilon,t}|\mathcal{F}_r]\right)dr\bigg|^4d\theta
\\
&\leq \frac{t^4}{\varepsilon^8}\int_0^t \int_{0}^tE\big|D_{\theta} D_rX_{\varepsilon,t} E[D_rX_{\varepsilon,t}|\mathcal{F}_r] + D_rX_{\varepsilon,t}E[D_{\theta}D_rX_{\varepsilon,t}|\mathcal{F}_r]\big|^4drd\theta\\
& \le \frac{8t^4}{\varepsilon^8}\int_0^t \int_{0}^t \left(E\left|D_{\theta} D_rX_{\varepsilon,t} E[D_rX_{\varepsilon,t}|\mathcal{F}_r]\right|^4+ E\left|D_rX_{\varepsilon,t}E[D_{\theta}D_rX_{\varepsilon,t}|\mathcal{F}_r]\right|^4\right)drd\theta\\
& \le \frac{16t^4}{\varepsilon^8}\int_0^t \int_{0}^t \sqrt{E|D_{\theta} D_rX_{\varepsilon,t}|^8 E|D_rX_{\varepsilon,t}|^8}drd\theta.
\end{align*}
Hence, recalling the estimates (\ref{ct1}) and (\ref{ct1a}), we get
$$E\| D\Theta_{\tilde{X}_{\varepsilon,t}}\|^4_{L^2[0, T]}\le C \varepsilon^{4}t^6\,\,\,\forall\,\varepsilon\in (0,1),t \in [0,T],$$
where $C>0$ is a finite constant not depending on $t$ and $\varepsilon.$ The proof of the proposition is complete.
\end{proof}

\noindent{\it Proof of Theorem \ref{klf9}.}  Put $\tilde{u}_r = E[D_r \tilde{X}_{\varepsilon,t} |\mathcal{F}_r]$ for $0\leq r\leq t\leq T.$ Then, we have
\begin{align*}
	||\tilde{u}||^8 _{L^2[0,T]}= \left(\int_0^t |E[D_r\tilde{X}_{\varepsilon,t}|\mathcal{F}_r]|^2dr\right)^4= \frac{1}{\varepsilon ^8}\left(\int_0^t |E[D_rX_{\varepsilon,t}|\mathcal{F}_r]|^2dr\right)^4.
\end{align*}
Using H\"{o}lder's inequality and the estimate (\ref{ct1}) we have
\begin{align}
E||\tilde{u}||^{8} _{L^2[0,T]}&\le \frac{t^3}{\varepsilon ^8}\int_0^tE\big |E[D_rX_{\varepsilon,t}|\mathcal{F}_r]\big|^8dr\notag
\\
&\le \frac{t^3}{\varepsilon ^8}\int_0^t E|D_rX_{\varepsilon,t}|^8dr\notag
\\
&\le Ct^4\,\,\,\forall\,\varepsilon\in (0,1),t \in [0,T].\label{ct7}
\end{align}
By applying Theorem \ref{lm2} to $F=\tilde{X}_{\varepsilon,t}$ and $N=N_t$ we get
\begin{align*}
I(\tilde{X}_{\varepsilon,t}\|N_t)\leq c\left(\frac{1}{\beta_t^4}(E[\tilde{X}_{\varepsilon,t}])^2+A_{\tilde{X}_{\varepsilon,t}}|{\rm Var}(\tilde{X}_{\varepsilon,t})-\beta_t^2|^2+C_{\tilde{X}_{\varepsilon,t}}\left(E\|D\Theta_{\tilde{X}_{\varepsilon,t}}\|^4\right)^{1/2}\right),
\end{align*}
where $c$ is an absolute constant and
\begin{align*}
A_{\tilde{X}_{\varepsilon,t}}:= \frac{1}{\beta_t^4}\left(E||\tilde{u}||^{8} _{L^2[0,T]}E|\Theta_{\tilde{X}_{\varepsilon,t}}|^{-8}\right)^{1/4},
C_{\tilde{X}_{\varepsilon,t}}:=A_{\tilde{X}_{\varepsilon,t}}+\left(E\|\tilde{u}\|_{L^2[0,T]}^8E|\Theta_{\tilde{X}_{\varepsilon,t}}|^{-16}\right)^{1/4}.
\end{align*}
Using the estimates (\ref{mmam1}) and (\ref{ct7}) we obtain
\begin{align*}
&A_{\tilde{X}_{\varepsilon,t}}\le  \frac{Ct}{\beta_t^4}\left(E\bigg|\int_0^t\sigma^2(r,X_{\varepsilon,r})dr\bigg|^{-p_0}\right)^{\frac{2}{p_0}},\\
&C_{\tilde{X}_{\varepsilon,t}}\le \frac{Ct}{\beta_t^4}\left(E\bigg|\int_0^t\sigma^2(r,X_{\varepsilon,r})dr\bigg|^{-p_0}\right)^{\frac{2}{p_0}}
+Ct\left(E\bigg|\int_0^t\sigma^2(r,X_{\varepsilon,r})dr\bigg|^{-p_0}\right)^{\frac{4}{p_0}}.
\end{align*}
Furthermore, thanks to Propositions \ref{pro1} and \ref{pr3.7}, we have
\begin{align*}
&|E[\tilde{X}_{\varepsilon,t}]|^2\leq Ct^4\varepsilon^2,\\
&|{\rm Var}(\tilde{X}_{\varepsilon,t})-\beta_t^2|^2\leq Ct^3\varepsilon^2,\\
&\left(E\|D\Theta_{\tilde{X}_{\varepsilon,t}}\|^4\right)^{1/2}\leq C t^3\varepsilon^2.
\end{align*}
Combining the above computations gives us
\begin{multline*}
I(\tilde{X}_{\varepsilon,t}\|N_t)\leq C\bigg(\frac{t^4}{\beta_t^4}+\frac{t^4}{\beta_t^4}\bigg(E\bigg|\int_0^t\sigma^2(r,X_{\varepsilon,r})dr\bigg|^{-p_0}\bigg)^{\frac{2}{p_0}}\\+
t^4\bigg(E\bigg|\int_0^t\sigma^2(r,X_{\varepsilon,r})dr\bigg|^{-p_0}\bigg)^{\frac{4}{p_0}}\bigg)\varepsilon^2.\hspace{2cm}
\end{multline*}
So the proof of Theorem \ref{klf9} is complete.\hfill$\square$
\subsection{Proof of Theorem \ref{k7zlf}}
Our idea is to use the following relation between the Fisher information and total variation distances, see \cite{Pinsker1964,Stam1959}:
\begin{equation}\label{mm}
\sqrt{I(\tilde{X}_{\varepsilon,t}\|N_t)}\geq d_{TV}(\tilde{X}_{\varepsilon,t}\|N_t):=\frac{1}{2}\sup\limits_{g}|E[g(\tilde{X}_{\varepsilon,t})]-E[g(N_t)]|,
\end{equation}
where the supremum is running over all measurable functions $g$ bounded by $1.$ Thus our main task is to find a lower bound for $\lim\limits_{\varepsilon\to 0}d_{TV}(\tilde{X}_{\varepsilon,t}\|N_t).$
\begin{proposition}Suppose the assumptions $(A_1)$-$(A_2).$ Then, for each $p \ge 2,$ we have
\begin{equation}\label{ccv}
\sup\limits_{t\in[0,T]}E|\tilde{X}_{\varepsilon,t}-U_t|^p\leq C\varepsilon^p \,\,\,\forall\,\varepsilon\in (0,1),
\end{equation}
\begin{equation}\label{cxcv}
\sup\limits_{t\in[0,T]}\int_0^t E|D_r\tilde{X}_{\varepsilon,t}-D_rU_t|^2dr\leq C\varepsilon^2\,\,\,\forall\,\varepsilon\in (0,1),
\end{equation}
where $C$ is a positive constant not depending on $\varepsilon.$
\end{proposition}
\begin{proof}Recalling (\ref{ydlq1}) and (\ref{ydlq}) we have
\begin{align}
\tilde{X}_{\varepsilon,t}-U_t=\int_0^tb'(s,x_s)(\tilde{X}_{\varepsilon,s}-U_s)ds
&+\frac{1}{2\varepsilon}\int_0^tb''(s,x_s+\theta_s(X_{\varepsilon,s}-x_s))(X_{\varepsilon,s}-x_s)^2ds\notag\\
&+\int_0^t(\sigma(s,X_{\varepsilon,s})-\sigma(s,x_s))dB_s,\,\,t\in [0,T].\label{vvn}
\end{align}
Hence, under the assumption $(A_2),$ we can use H\"{o}lder and Burkholder-Davis-Gundy inequalities to get
\begin{align*}
E|\tilde{X}_{\varepsilon,t}-U_t|^p\leq C\int_0^tE|\tilde{X}_{\varepsilon,s}-U_s|^pds
&+\frac{C}{\varepsilon^p}\int_0^tE|X_{\varepsilon,s}-x_s|^{2p}ds\\
&+C\int_0^tE|X_{\varepsilon,s}-x_s|^{p}ds,\,\,t\in [0,T]
\end{align*}
for all $\varepsilon\in (0,1)$ and for some $C>0.$ Consequently, it follows from the estimate (\ref{oo1}) that
$$E|\tilde{X}_{\varepsilon,t}-U_t|^p\leq C\varepsilon^p+C\int_0^tE|\tilde{X}_{\varepsilon,s}-U_s|^pds,\,\,t\in [0,T],$$
where $C$ is a positive constant not depending on $t$ and $\varepsilon.$ So we obtain the estimate (\ref{ccv}) by using Gronwall's lemma. Indeed, we have
$$E|\tilde{X}_{\varepsilon,t}-U_t|^p\leq C\varepsilon^p e^{Ct}\leq C\varepsilon^p \,\,\,\forall\,\varepsilon\in (0,1),t \in [0,T].$$
Let us now prove (\ref{cxcv}). We have, for $0\leq r \leq T,$
\begin{align*}
D_r\tilde{X}_{\varepsilon,t}-D_rU_t&=\frac{1}{\varepsilon}D_rX_{\varepsilon,t}-D_rU_t\\
&=\sigma(r,X_{\varepsilon,r})e^{\int_r^tb'(u,X_{\varepsilon,u})du}Z_{r,t}-\sigma(r,x_r)e^{\int_r^tb'(u,x_r)du},
\end{align*}
where $Z_{r,t}$ is defined by (\ref{alne}). We deduce
\begin{align*}
|D_r\tilde{X}_{\varepsilon,t}-D_rU_t|
&\leq |\sigma(r,X_{\varepsilon,r})-\sigma(r,x_r)|e^{\int_r^tb'(u,X_{\varepsilon,u})du}Z_{r,t}\\
&+\sigma(r,x_r)\left|e^{\int_r^tb'(u,X_{\varepsilon,u})du}-e^{\int_r^tb'(u,x_u)du}\right|Z_{r,t}+\sigma(r,x_r)e^{\int_r^tb'(u,x_u)du}|Z_{r,t}-1|
\end{align*}
and
\begin{align*}
\int_0^t E|D_r\tilde{X}_{\varepsilon,t}-D_rU_t|^2dr&
\leq 3\int_0^tE\left[|\sigma(r,X_{\varepsilon,r})-\sigma(r,x_r)|^2e^{2\int_r^tb'(u,X_{\varepsilon,u})du}Z_{r,t}^2\right]dr\\
&+3\int_0^t \sigma^2(r,x_r)E\left[\left|e^{\int_r^tb'(u,X_{\varepsilon,u})du}-e^{\int_r^tb'(u,x_u)du}\right|^2Z_{r,t}^2\right]dr\\&+3\int_0^t \sigma^2(r,x_r)e^{2\int_r^tb'(u,x_u)du}E|Z_{r,t}-1|^2dr.
\end{align*}
Hence, using the same arguments as in the proof (\ref{nvn}) and  (\ref{nvn1}), we get
\begin{align*}
\int_0^t E|D_r\tilde{X}_{\varepsilon,t}-D_rU_t|^2dr&\leq C\varepsilon^2+3\int_0^t \sigma^2(r,x_r)e^{2\int_r^tb'(u,x_u)du}E|Z_{r,t}-1|^2dr.
\end{align*}
On the other hand, it follows from the equation (\ref{axlne}) that
$$\sup\limits_{0\leq r\leq t\leq T}E|Z_{r,t}-1|^2\leq C\varepsilon^2\,\,\,\forall\,\varepsilon\in (0,1).$$
So we obtain
$$\int_0^t E|D_r\tilde{X}_{\varepsilon,t}-D_rU_t|^2dr\leq C\varepsilon^2\,\,\,\forall\,\varepsilon\in (0,1),t \in [0,T],$$
where $C$ is a positive constant not depending on $t$ and $\varepsilon.$ This finishes the proof of the proposition.
\end{proof}
\begin{proposition}Suppose the assumptions $(A_1)$-$(A_2).$ Then, for each $p\geq 2,$ we have
\begin{equation}\label{avn03}
\lim\limits_{\varepsilon\to0}E\bigg|\frac{\tilde{X}_{\varepsilon,t}-U_t}{\varepsilon}-V_t\bigg|^p=0,\,\,t\in[0,T].
\end{equation}
\end{proposition}
\begin{proof} We rewrite (\ref{vvn}) as follows
\begin{align*}
\tilde{X}_{\varepsilon,t}-U_t=\int_0^tb'(s,x_s)(\tilde{X}_{\varepsilon,s}-U_s)ds
&+\frac{1}{2\varepsilon}\int_0^tb''(s,x_s+\theta_s(X_{\varepsilon,s}-x_s))(X_{\varepsilon,s}-x_s)^2ds\notag\\
&+\int_0^t\sigma'(s,x_s+\theta_s(X_{\varepsilon,s}-x_s))(X_{\varepsilon,s}-x_s)dB_s,\,\,t\in [0,T].
\end{align*}
where, for each $0\leq s\leq t,$ $\theta_s,\eta_s$ are random variables lying between $0$ and $1.$ Hence, by the definition (\ref{ydlq2}) of $V_t,$ we obtain
\begin{align*}
\frac{\tilde{X}_{\varepsilon,t}-U_t}{\varepsilon}-V_t&=\int_0^tb'(s,x_s)\left(\frac{\tilde{X}_{\varepsilon,s}-U_s}{\varepsilon}-V_s\right)ds\\
&+\frac{1}{2}\int_0^t\big(b''(s,x_s+\theta_s(X_{\varepsilon,s}-x_s))\tilde{X}_{\varepsilon,s}^2-b''(s,x_s)U_s^2\big)ds\\
&+\int_0^t\big(\sigma'(s,x_s+\eta_s(X_{\varepsilon,s}-x_s))\tilde{X}_{\varepsilon,s}-\sigma'(s,x_s)U_s\big)dB_s,\,\,t\in [0,T].
\end{align*}
As a consequence, we deduce
$$E\bigg|\frac{\tilde{X}_{\varepsilon,t}-U_t}{\varepsilon}-V_t\bigg|^p\leq C\int_0^tE\bigg|\frac{\tilde{X}_{\varepsilon,s}-U_s}{\varepsilon}-V_s\bigg|^pds+CK_{\varepsilon},\,\,t\in [0,T]$$
 for some $C>0$ and for all $\varepsilon\in (0,1),$ where $K_{\varepsilon}$ is given by
\begin{align*}
K_{\varepsilon}:=\int_0^TE\big|b''(s,x_s&+\theta_s(X_{\varepsilon,s}-x_s))\tilde{X}_{\varepsilon,s}^2-b''(s,x_s)U_s^2\big|^pds\\
&+\int_0^TE\big|\sigma'(s,x_s+\eta_s(X_{\varepsilon,s}-x_s))\tilde{X}_{\varepsilon,s}-\sigma'(s,x_s)U_s\big|^pds.
\end{align*}
An application of Gronwall's lemma gives us
\begin{equation}\label{avn03z}
E\bigg|\frac{\tilde{X}_{\varepsilon,t}-U_t}{\varepsilon}-V_t\bigg|^p\leq CK_{\varepsilon}e^{Ct}\leq CK_{\varepsilon}\,\,\,\forall\,\varepsilon\in (0,1),t \in [0,T].
\end{equation}
We observe that
\begin{align*}
\int_0^T&E\big|b''(s,x_s+\theta_s(X_{\varepsilon,s}-x_s))\tilde{X}_{\varepsilon,s}^2-b''(s,x_s)U_s^2\big|^pds\\
&\leq 2^{p-1}\int_0^TE\big|b''(s,x_s+\theta_s(X_{\varepsilon,s}-x_s))(\tilde{X}_{\varepsilon,s}^2-U_s^2)\big|^pds\\
&+2^{p-1}\int_0^TE\big|(b''(s,x_s+\theta_s(X_{\varepsilon,s}-x_s))-b''(s,x_s))U_s^2\big|^pds\\
&\leq 2^{p-1}L^p\int_0^T\sqrt{E|\tilde{X}_{\varepsilon,s}-U_s|^{2p}E|\tilde{X}_{\varepsilon,s}+U_s|^{2p}}ds\\
&+2^{p-1}\int_0^T\sqrt{E|b''(s,x_s+\theta_s(X_{\varepsilon,s}-x_s))-b''(s,x_s)|^{2p}E|U_s|^{4p}}ds.
\end{align*}
Furthermore, $E|U_s|^{q}<\infty$ for all $q>1.$ Hence, by the estimate (\ref{ccv}) and the dominated convergence theorem, we get
$$\int_0^TE\big|b''(s,x_s+\theta_s(X_{\varepsilon,s}-x_s))\tilde{X}_{\varepsilon,s}^2-b''(s,x_s)U_s^2\big|^pds\to 0\,\,\text{as}\,\,\varepsilon\to0.$$
Similarly, we also have
$$\int_0^TE\big|\sigma'(s,x_s+\eta_s(X_{\varepsilon,s}-x_s))\tilde{X}_{\varepsilon,s}-\sigma'(s,x_s)U_s\big|^pds\to 0\,\,\text{as}\,\,\varepsilon\to0.$$
Those imply that $K_{\varepsilon}\to$ as $\varepsilon\to0.$ So the desired relation (\ref{avn03}) follows from (\ref{avn03z}). The proof of the proposition is complete.
\end{proof}

\begin{proposition}Suppose the assumptions $(A_1)$-$(A_2).$ Then, for any continuous and bounded function $g$ and for each $t\in (0,T],$ we have
\begin{equation}\label{mm0}
\lim\limits_{\varepsilon\to0}\frac{E[g(\tilde{X}_{\varepsilon,t})]-E[g(U_t)]}{\varepsilon}=\frac{1}{\beta_t^2}E[g(U_t)\delta(V_tDU_t)].
\end{equation}
\end{proposition}
\begin{proof}For simplicity, we write $\langle.,. \rangle$ instead of $\langle.,. \rangle_{L^2[0,T]}$ and $\|.\|$ instead of $\|.\|_{L^2[0,T]}.$ Fix $t\in (0,T],$ by using the formula (3.2) in \cite{DungNT2022}, we get
\begin{align*}
E[g(\tilde{X}_{\varepsilon,t})]-E[g(U_t)]&=E\left[\int_{U_t}^{\tilde{X}_{\varepsilon,t}} g(z)dz\delta\bigg(\frac{DU_t}{\|DU_t\|^2}\bigg)\right]-E\left[\frac{g(\tilde{X}_{\varepsilon,t}) \langle D\tilde{X}_{\varepsilon,t}-DU_t, DU_t \rangle}{\|DU_t\|^2}\right].
\end{align*}
We observe that $\|DU_t\|^2=\beta_t^2$ and $\delta\left(\frac{DU_t}{\|DU_t\|^2}\right)=U_t/\beta_t^2.$ So we obtain
\begin{align*}
E[g(\tilde{X}_{\varepsilon,t})]-E[g(U_t)]&=\frac{1}{\beta_t^2}E\left[U_t\int_{U_t}^{\tilde{X}_{\varepsilon,t}} g(z)dz\right]-\frac{1}{\beta_t^2}E\left[g(\tilde{X}_{\varepsilon,t}) \langle D\tilde{X}_{\varepsilon,t}-DU_t, DU_t \rangle\right].
\end{align*}
Then, for $\varepsilon\neq 0,$
\begin{align}
&\frac{E[g(\tilde{X}_{\varepsilon,t})]-E[g(U_t)]}{\varepsilon}-\frac{1}{\beta_t^2}E\left[g(U_t)V_tU_t\right]+\frac{1}{\beta_t^2}E\left[g(U_t)\langle DV_t, DU_t \rangle\right]\notag\\
&=\frac{1}{\beta_t^2}E\left[\bigg(\frac{1}{\varepsilon}\int_{U_t}^{\tilde{X}_{\varepsilon,t}} g(z)dz-g(U_t)V_t\bigg)U_t\right]-\frac{1}{\beta_t^2}E\left[(g(\tilde{X}_{\varepsilon,t}) -g(U_t))\bigg\langle \frac{D\tilde{X}_{\varepsilon,t}-DU_t}{\varepsilon}, DU_t\bigg\rangle\right]\notag\\
&-\frac{1}{\beta_t^2}E\left[g(U_t)\bigg\langle \frac{D\tilde{X}_{\varepsilon,t}-DU_t}{\varepsilon}-DV_t, DU_t\bigg\rangle\right],\,\,0<t\leq T.\label{jaq}
\end{align}
We observe that
\begin{align*}
\frac{1}{\varepsilon}&\int_{U_t}^{\tilde{X}_{\varepsilon,t}} g(z)dz-g(U_t)V_t=\frac{\tilde{X}_{\varepsilon,t}-U_t}{\varepsilon}\int_0^1 g(U_t+z(\tilde{X}_{\varepsilon,t}-U_t))dz-g(U_t)V_t\\
&=\bigg(\frac{\tilde{X}_{\varepsilon,t}-U_t}{\varepsilon}-V_t\bigg)\int_0^1 g(U_t+z(\tilde{X}_{\varepsilon,t}-U_t))dz+V_t\int_0^1 (g(U_t+z(\tilde{X}_{\varepsilon,t}-U_t))-g(U_t))dz,
\end{align*}
and hence,
\begin{align*}
E\bigg|\bigg(\frac{1}{\varepsilon}\int_{U_t}^{\tilde{X}_{\varepsilon,t}} g(z)dz-g(U_t)V_t\bigg)&U_t\bigg|\leq \|g\|_\infty E\left|\bigg(\frac{\tilde{X}_{\varepsilon,t}-U_t}{\varepsilon}-V_t\bigg)U_t\right|\\
&+E\left|V_tU_t\int_0^1 (g(U_t+z(\tilde{X}_{\varepsilon,t}-U_t))-g(U_t))dz\right|.
\end{align*}
Because the random variables $U_t$ and $V_t$ belong to $L^2(\Omega,)$ we have
$$\lim\limits_{\varepsilon\to 0}E\left|\bigg(\frac{\tilde{X}_{\varepsilon,t}-U_t}{\varepsilon}-V_t\bigg)U_t\right|=0\,\,\,\text{by the estimate (\ref{avn03}),}$$
and
$$\lim\limits_{\varepsilon\to 0}E\left|V_tU_t\int_0^1 (g(U_t+z(\tilde{X}_{\varepsilon,t}-U_t))-g(U_t))dz\right|=0$$
by the dominated convergence theorem. So it holds that
\begin{equation}\label{jaq1}\lim\limits_{\varepsilon\to 0}E\bigg[\bigg(\frac{1}{\varepsilon}\int_{U_t}^{\tilde{X}_{\varepsilon,t}} g(z)dz-g(U_t)V_t\bigg)U_t\bigg]=0.
\end{equation}
On the other hand, we have
\begin{align*}
E&\left[(g(\tilde{X}_{\varepsilon,t}) -g(U_t))\bigg\langle \frac{D\tilde{X}_{\varepsilon,t}-DU_t}{\varepsilon}, DU_t\bigg\rangle\right]\leq \frac{1}{\beta_t} E\left[\frac{|g(\tilde{X}_{\varepsilon,t}) -g(U_t)|\| D\tilde{X}_{\varepsilon,t}-DU_t\|}{\varepsilon}\right]\\
&\leq \frac{1}{\beta_t}(E|g(\tilde{X}_{\varepsilon,t}) -g(U_t)|^2)^{\frac{1}{2}}\bigg(\frac{E\| D\tilde{X}_{\varepsilon,t}-DU_t\|^2}{\varepsilon^2}\bigg)^{\frac{1}{2}}.
\end{align*}
Once again, by (\ref{cxcv}) and the dominated convergence theorem, we derive
\begin{equation}\label{jaq2}
\lim\limits_{\varepsilon\to 0}E\left[(g(\tilde{X}_{\varepsilon,t}) -g(U_t))\bigg\langle \frac{D\tilde{X}_{\varepsilon,t}-DU_t}{\varepsilon}, DU_t\bigg\rangle\right]=0.
\end{equation}
In view of Lemma 1.2.3 in \cite{nualartm2}, it follows from (\ref{cxcv}) and (\ref{avn03}) that
\begin{equation}\label{jaq3}
\lim\limits_{\varepsilon\to 0}E\left[g(U_t)\bigg\langle \frac{D\tilde{X}_{\varepsilon,t}-DU_t}{\varepsilon}-DV_t, DU_t\bigg\rangle\right]=0.
\end{equation}
Combining (\ref{jaq})-(\ref{jaq3}) yields
$$\lim\limits_{\varepsilon\to 0}\frac{E[g(\tilde{X}_{\varepsilon,t})]-E[g(U_t)]}{\varepsilon}=\frac{1}{\beta_t^2}E\left[g(U_t)V_tU_t\right]
-\frac{1}{\beta_t^2}E\left[g(U_t)\langle DV_t, DU_t \rangle\right].$$
Then we obtain (\ref{mm0}) by using the duality relationships (\ref{ct2}). This finishes the proof of the proposition.
\end{proof}

\noindent{\it Proof of Theorem \ref{k7zlf}.} We first note that, for each $t\in(0,T],$ $U_t$ is also a normal random variable with mean zero and variance $\beta_t^2.$ This means that $I(\tilde{X}_{\varepsilon,t}\|N_t)=I(\tilde{X}_{\varepsilon,t}\|U_t).$ Hence, it follows from the relations (\ref{mm}) and (\ref{mm0}) that
\begin{align*}
\lim\limits_{\varepsilon\to0}\sqrt{I(\tilde{X}_{\varepsilon,t}\|N_t)}\geq \frac{1}{2\beta_t^2}\left|E\left[g(U_t)\delta(V_tDU_t)\right]\right|
=\frac{1}{2\beta_t^2}\left|E\left[g(U_t)E[\delta(V_tDU_t)|U_t]\right]\right|
\end{align*}
for any continuous function $g$ bounded by $1.$ Then, by the routine approximation argument, we can choose to use $g(x)={\rm sign}\left(E\left[\delta\left(V_tDU_t\right)\big|U_t=x\right]\right)$ and we obtain (\ref{mm1}).

The proof of Theorem \ref{k7zlf} is complete.  \hfill$\square$
\subsection{Proof of Theorem \ref{theorem2}}
\noindent We recall that
$$Y_{\varepsilon,t} = \int_0^t  f(s, X_{\varepsilon,s})ds,\,\,t\in[0,T].$$
By the chain rule for Malliavin derivatives, we have
\begin{align}
	D_rY_{\varepsilon,t}=& \int_r^tf'(s,X_{\varepsilon,s})D_rX_{\varepsilon,s}ds,\label{dhc1}
\\	D_{\theta} D_rY_{\varepsilon,t} =& \int_{\theta \vee r}^t\left[f''(s, X_{\varepsilon,s})D_{\theta}X_{\varepsilon,s}D_rX_{\varepsilon,s} + f'(s, X_{\varepsilon,s})D_{\theta}D_	rX_{\varepsilon,s}\right]ds\label{dhc2}
\end{align}
for all $0\leq r,\theta\leq t\le T,$ where $f'(t,x)=\frac{\partial f(t,x)}{\partial{x}},\mbox{  }f''(t,x)=\frac{\partial ^2f(t,x)}{\partial x^2}.$

For the proof of Theorem \ref{theorem2} we will need the following Propositions \ref{pr7}-\ref{rsprx}.
\begin{proposition}\label{pr7}
Suppose the assumptions $(A_1)$-$(A_3).$ Then, we have
\begin{align}
\sup\limits_{0\le r \le t}E|D_rY_{\varepsilon,t}|^p&\le Ct^p\varepsilon^p\,\,\,\forall\,\varepsilon\in (0,1),t\in[0,T], \label{df}\\
\sup\limits_{0\le r,\theta \le t}E|D_{\theta} D_rY_{\varepsilon,t}|^p&\le Ct^p\varepsilon^{2p}\,\,\,\forall\,\varepsilon\in (0,1),t\in[0,T],\label{ddf}
\end{align}
where $C$ is a positive constant not depending on $t$ and $\varepsilon.$
\end{proposition}
\begin{proof} We first note that, by the polynomial growth property of $f',f''$ and the estimate (\ref{s4hd}), we have
\begin{equation}\label{oxzo1}
\sup\limits_{0\leq s\leq T}E|f'(s,X_{\varepsilon,s})|^{2p}+\sup\limits_{0\leq s\leq T}E|f''(s,X_{\varepsilon,s})|^{2p}\leq C
\end{equation}
 for all $\varepsilon\in (0,1),$ where $C$ is a positive constant not depending on $\varepsilon.$

By using H\"older's inequality, it follows from (\ref{dhc1}) that
	\begin{align*}
		E|D_rY_{\varepsilon,t}|^p &=E\left|\int_r^t f'(s,X_{\varepsilon,s})D_rX_{\varepsilon,s}ds \right|^{p}\\
		& \le  (t-r)^{p-1} \int_r^tE|f'(s,X_{\varepsilon,s})D_rX_{\varepsilon,s}|^{p}ds\\
& \le  t^{p-1} \int_r^t\sqrt{E|f'(s,X_{\varepsilon,s})|^{2p}E|D_rX_{\varepsilon,s}|^{2p}}ds\\
& \le  C t^{p-1}\int_r^t\sqrt{E|D_rX_{\varepsilon,s}|^{2p}}ds,
	\end{align*}
which, together with the estimate (\ref{ct1}),  implies (\ref{df}).	The estimate (\ref{ddf}) can be proved similarly. Indeed, we obtain from (\ref{dhc2}) that
\begin{align*}
		E|D_{\theta}& D_rY_{\varepsilon,t}|^p \le  2^{p-1} \left(E\left| \int_{\theta \vee r}^t f''(s, X_{\varepsilon,s})D_{\theta}X_{\varepsilon,s}D_rX_{\varepsilon,s}ds\right|^p+ E\left|\int_{\theta \vee r}^tf'(s, X_{\varepsilon,s})D_{\theta}D_rX_{\varepsilon,s}ds\right|^p\right)\\
&\leq 2^{p-1}t^{p-1} \left(\int_{\theta \vee r}^t E|f''(s, X_{\varepsilon,s})D_{\theta}X_{\varepsilon,s}D_rX_{\varepsilon,s}|^pds+\int_{\theta \vee r}^tE|f'(s, X_{\varepsilon,s})D_{\theta}D_rX_{\varepsilon,s}|^pds\right)\\
&\leq 2^{p-1}t^{p-1} \bigg(\int_{\theta \vee r}^t \sqrt[3]{E|f''(s, X_{\varepsilon,s})|^{3p}E|D_{\theta}X_{\varepsilon,s}|^{3p}E|D_rX_{\varepsilon,s}|^{3p}}ds\\
&\hspace{5cm}+\int_{\theta \vee r}^t\sqrt{E|f'(s, X_{\varepsilon,s})|^{2p}E|D_{\theta}D_rX_{\varepsilon,s}|^{2p}}ds\bigg)\\
&\leq Ct^{p-1} \bigg(\int_{\theta \vee r}^t \sqrt[3]{E|D_{\theta}X_{\varepsilon,s}|^{3p}E|D_rX_{\varepsilon,s}|^{3p}}ds+\int_{\theta \vee r}^t\sqrt{E|D_{\theta}D_rX_{\varepsilon,s}|^{2p}}ds\bigg).
	\end{align*}
So we obtain (\ref{ddf}) by using the estimates (\ref{ct1}) and (\ref{ct1a}). The proof of the proposition is complete.
\end{proof}

\begin{proposition}\label{pr8}Suppose the assumptions $(A_1)$-$(A_3).$ Let $(\tilde{Y}_{\varepsilon,t})_{t\in[0,T]}$ be as in Theorem \ref{theorem2}.
Then, we have
	\begin{align}
		|E[\tilde{Y}_{\varepsilon,t}]|&\le Ct^2\varepsilon\,\,\,\forall\,\varepsilon\in (0,1),t\in[0,T],\label{yy1}\\
		|{\rm Var}(\tilde{Y}_{\varepsilon,t})-\gamma_t^2|&\leq Ct^{7/2}\varepsilon\,\,\,\forall\,\varepsilon\in (0,1),t\in[0,T],\label{yy2}
	\end{align}
where $C$ is a positive constant not depending on $t$ and $\varepsilon.$
\end{proposition}
\begin{proof} We recall that
$$  \tilde{Y}_{\varepsilon,t}=\frac{1}{\varepsilon}\int_0^{t}(f(s,X_{\varepsilon,s})-f(s,x_s))ds,\,\,t\in [0,T].$$
Hence, in view of the estimate (\ref{ozo1}), we obtain
	\begin{align*}
		|E[\tilde{Y}_{\varepsilon,t}]|&\leq \frac{1}{\varepsilon}\int_0^{t}E|f(s,X_{\varepsilon,s})-f(s,x_s)|ds\\
&\leq \frac{1}{\varepsilon}\int_0^{t}\sqrt{E|f(s,X_{\varepsilon,s})-f(s,x_s)|^2}ds\\
&\leq Ct^2\varepsilon.
	\end{align*}
This completes the proof of (\ref{yy1}). Let us prove the estimate (\ref{yy2}).	By the Clark-Ocone formula and the It\^o isometry we have
	\begin{align*}
		{\rm Var}(\tilde{Y}_{\varepsilon,t})&=E\left[\int_0^T E[D_r\tilde{Y}_{\varepsilon,t}|\mathcal{F}_r]dB_r\right]^2=\frac{1}{\varepsilon^2}E\left[\int_0^T(E[D_rY_{\varepsilon,t}|\mathcal{F}_r])^2dr\right]\\
		&=\frac{1}{\varepsilon^2}E\left[\int_0^t \left(E\left[\int_r^tf'(s,X_{\varepsilon,s})D_rX_{\varepsilon,s}ds\Big|\mathcal{F}_r\right]\right)^2dr\right].
	\end{align*}
	Then, we obtain
	\begin{align}
		&{\rm Var}(\tilde{Y}_{\varepsilon,t})-\gamma_t^2\notag\\
		&=\frac{1}{\varepsilon^2}E\left[\int_0^t \left(\int_r^tE\left[f'(s,X_{\varepsilon,s})D_rX_{\varepsilon,s}\big|\mathcal{F}_r\right]ds\right)^2dr\right]-\int_0^t \left(\int_r^tf'(s,x_s)\sigma (r,x_r)e ^{\int_r^s b'(u,x_u)du}ds\right)^2dr\notag\\
		&=\int_0^t E[H_{t,r}G_{t,r}]dr,\label{ant1}
	\end{align}
	where, for $0\leq r\leq t\leq T,$ $H_{t,r}$ and $G_{t,r}$ are defined by
	$$H_{t,r}:=\frac{1}{\varepsilon}\int_r^tE\left[f'(s,X_{\varepsilon,s})D_rX_{\varepsilon,s}\big|\mathcal{F}_r\right]ds-
\int_r^tf'(s,x_s)\sigma (r,x_r)e ^{\int_r^s b'(u,x_u)du}ds,$$
	$$G_{t,r}:=\frac{1}{\varepsilon}\int_r^tE\left[f'(s,X_{\varepsilon,s})D_rX_{\varepsilon,s}\big|\mathcal{F}_r\right]ds+
\int_r^tf'(s,x_s)\sigma (r,x_r)e ^{\int_r^s b'(u,x_u)du}ds.$$
Recalling the representation (\ref{jknm})	we have
\begin{align*} H_{t,r}&:=(\sigma(r,X_{\varepsilon,r})-\sigma(r,x_r))\int_r^tE\left[f'(s,X_{\varepsilon,s})e^{\int_r^sb'(u,X_{\varepsilon,u})}Z_{r,s}\big|\mathcal{F}_r\right]ds\\
		&+\sigma(r,x_r)\int_r^tE\left[(f'(s,X_{\varepsilon,s})-f'(s,x_s))e ^{\int_r^s b'(u,x_u)du}Z_{r,s}\big|\mathcal{F}_r\right]ds\\
		&+\sigma(r,x_r)\int_r^tf'(s,x_s)E\left[\big(e^{\int_r^sb'(u,X_{\varepsilon,u})}-e ^{\int_r^s b'(u,x_u)du}\big)Z_{r,s}\big|\mathcal{F}_r\right]ds
	\end{align*}
and hence, note that $b'$ is bounded by $L$ and $\sup\limits_{0\leq r\leq T}|\sigma(r,x_r)|+\sup\limits_{0\leq r\leq T}|f'(r,x_r)|<\infty,$  we deduce
\begin{align*} &|H_{t,r}|\leq C|\sigma(r,X_{\varepsilon,r})-\sigma(r,x_r)|\int_r^tE\left[|f'(s,X_{\varepsilon,s})|Z_{r,s}\big|\mathcal{F}_r\right]ds\\
		&+C\int_r^tE\left[|f'(s,X_{\varepsilon,s})-f'(s,x_s)|Z_{r,s}\big|\mathcal{F}_r\right]ds+C\int_r^tE\left[\int_r^s|b'(u,X_{\varepsilon,u})- b'(u,x_u)|duZ_{r,s}\big|\mathcal{F}_r\right]ds.
\end{align*}
By using the Cauchy-Schwarz inequality we get
\begin{align*} |H_{t,r}|^2&\leq Ct|\sigma(r,X_{\varepsilon,r})-\sigma(r,x_r)|^2\int_r^tE\left[|f'(s,X_{\varepsilon,s})|^2Z_{r,s}^2\big|\mathcal{F}_r\right]ds\\
		&+Ct\int_r^tE\left[|f'(s,X_{\varepsilon,s})-f'(s,x_s)|^2Z_{r,s}^2\big|\mathcal{F}_r\right]ds\\
&+Ct^2\int_r^tE\left[\int_r^s|b'(u,X_{\varepsilon,u})- b'(u,x_u)|^2duZ_{r,s}^2\big|\mathcal{F}_r\right]ds
\end{align*}
and
\begin{align*} &E|H_{t,r}|^2\leq Ct^{\frac{3}{2}}\sqrt{E|\sigma(r,X_{\varepsilon,r})-\sigma(r,x_r)|^4\int_r^tE\left[|f'(s,X_{\varepsilon,s})|^4Z_{r,s}^4\right]ds}\\
		&+Ct\int_r^tE\left[|f'(s,X_{\varepsilon,s})-f'(s,x_s)|^2Z_{r,s}^2\right]ds+Ct^2\int_r^t\int_r^sE\left[|b'(u,X_{\varepsilon,u})- b'(u,x_u)|^2Z_{r,s}^2\right]duds\\
&\leq Ct^{\frac{3}{2}}\sqrt{E|\sigma(r,X_{\varepsilon,r})-\sigma(r,x_r)|^4\int_r^t\sqrt{E|f'(s,X_{\varepsilon,s})|^8E|Z_{r,s}|^8}ds}\\
		&+Ct\int_r^t\sqrt{E|f'(s,X_{\varepsilon,s})-f'(s,x_s)|^4E|Z_{r,s}|^4}ds\\
&+Ct^2\int_r^t\int_r^s\sqrt{E|b'(u,X_{\varepsilon,u})- b'(u,x_u)|^4E|Z_{r,s}|^4}duds.
\end{align*}
In view of the estimates (\ref{ozo1}), (\ref{iif2}) and (\ref{oxzo1}), we therefore obtain
\begin{equation}\label{ant2}
E|H_{t,r}|^2\leq Ct^3\varepsilon^2\,\,\,\forall\,\varepsilon\in (0,1),0\leq r\leq t\leq T,
\end{equation}
where $C$ is a positive constant not depending on $t$ and $\varepsilon.$ On the other hand, we have
\begin{align}
E|G_{t,r}|^2&=E\bigg|H_{t,r}+2\int_r^tf'(s,x_s)\sigma (r,x_r)e ^{\int_r^s b'(u,x_u)du}ds\bigg|^2\notag\\
&\leq 2E|H_{t,r}|^2+8\bigg|2\int_r^tf'(s,x_s)\sigma (r,x_r)e ^{\int_r^s b'(u,x_u)du}ds\bigg|^2\notag\\
&\leq C t^2\,\,\,\forall\,\varepsilon\in (0,1),0\leq r\leq t\leq T.\label{ant3}
\end{align}
Combining (\ref{ant1}), (\ref{ant2}) and (\ref{ant3}) gives us
\begin{align*}
		|{\rm Var}(\tilde{Y}_{\varepsilon,t})-\sigma_t^2|&= \int_0^tE|H_{t,r}G_{t,r}|dr\\
		&\leq \int_0^t\sqrt{E|H_{t,r}|^2E|G_{t,r}|^2}dr\\
&\le C t^{\frac{7}{2}}\varepsilon\,\,\,\forall\,\varepsilon\in (0,1),t\in[0,T].
\end{align*}
	This finishes the proof of (\ref{yy2}). The proof of the proposition is complete.
\end{proof}
\begin{proposition}\label{rspr8}Let $(\tilde{Y}_{\varepsilon,t})_{t\in[0,T]}$ be as in Theorem \ref{theorem2}. Define
$$\Theta_{\tilde{Y}_{\varepsilon,t}}:=\int_0^tD_r\tilde{Y}_{\varepsilon,t}E[D_r\tilde{Y}_{\varepsilon,t}|\mathcal{F}_r]dr,\,\,t\in [0,T]. $$
Then, under the assumption of Theorem \ref{theorem2}, we have
\begin{align}\label{momenam}
E|\Theta_{\tilde{Y}_{\varepsilon,t}}|^{-p}\le C\left(E\bigg|\int_0^t(t-r)^2\sigma^2(r,X_{\varepsilon,r})dr\bigg|^{-p_0}\right)^{\frac{p}{p_0}}\,\,\,\forall\,\varepsilon\in (0,1),t \in (0,T],
\end{align}
where $0<p<p_0$ and $C>0$ is a positive constant not depending on $t$ and $\varepsilon.$

\end{proposition}
\begin{proof}We have
\begin{align*}
	\Theta_{\tilde{Y}_{\varepsilon,t}}&=\frac{1}{\varepsilon^2}\int_0^tD_rY_{\varepsilon,t}E[D_rY_{\varepsilon,t}|\mathcal{F}_r]dr\\
& = \frac{1}{\varepsilon^2}\int_0^t \bigg(\int_r^tf'(s, X_{\varepsilon,s})D_rX_{\varepsilon,s}ds\bigg)\bigg(\int_r^tE\left[f'(s, X_{\varepsilon,s})D_rX_{\varepsilon,s}|\mathcal{F}_r\right]ds\bigg)dr.
\end{align*}
Note that $\|f'\|_0:=\inf\limits_{(t,x)}f'(t,x)>0.$ Hence, by using the same arguments as in the proof of Proposition \ref{momenam1}, we obtain
\begin{align*}
	\Theta_{\tilde{Y}_{\varepsilon,t}}& \geq  \|f'\|^2_0 e^{-\frac{5LT}{2}}e^{-2\max\limits_{0\le t\le T}M_t}\int_0^t(t-r)^2\sigma^2(r,X_{\varepsilon,r})dr
\end{align*}
and, for $0<p<p_0,$ we get
\begin{align*}
	E|\Theta_{\tilde{Y}_{\varepsilon,t}}|^{-p}
&\leq C\left(E\bigg|\int_0^t(t-r)^2\sigma^2(r,X_{\varepsilon,r})dr\bigg|^{-p_0}\right)^{\frac{p}{p_0}}\,\,\,\forall\,\varepsilon\in (0,1),t \in (0,T],
\end{align*}
where $C>0$ is a positive constant not depending on $t$ and $\varepsilon.$ The proof of the proposition is complete.
\end{proof}
\begin{proposition}\label{rsprx}Let $(\Theta_{\tilde{Y}_{\varepsilon,t}})_{t\in[0,T]}$ be as in Proposition \ref{rspr8}. Suppose the assumptions $(A_1)$-$(A_3).$ Then, we have
\begin{align}\label{Dgamma}
	E\|D\Theta_{\tilde{Y}_{\varepsilon,t}}\|_{L^2[0, T]}^4 \le C t^{14}\varepsilon^4\,\,\,\forall\,\varepsilon\in (0,1),t \in [0,T],
\end{align}
where $C$ is a positive constant not depending on $t$ and $\varepsilon.$
\end{proposition}

\begin{proof} The Malliavin derivative of $\Theta_{\tilde{Y}_{\varepsilon,t}}$ can be computed by
\begin{align*}
	D_{\theta}\Theta_{\tilde{Y}_{\varepsilon,t}} =\frac{1}{\varepsilon^2}\int_{0}^t\left(D_{\theta} D_rY_{\varepsilon,t} E[D_rY_{\varepsilon,t}|\mathcal{F}_r] + D_rY_{\varepsilon,t}E[D_{\theta}D_rY_{\varepsilon,t}|\mathcal{F}_r]\right)dr,\,\,0\leq \theta\leq t.
\end{align*}
Hence,
\begin{align*}
	\|D\Theta_{\tilde{Y}_{\varepsilon,t}}\|^2_{L^2[0, T]}&= \int_0^T|D_{\theta}\Theta_{\tilde{Y}_{\varepsilon,t}}|^2 d\theta\\
	&=\int_0^t\frac{1}{\varepsilon^4} \left(\int_{0}^t\left(D_{\theta} D_rY_{\varepsilon,t} E[D_rY_{\varepsilon,t}|\mathcal{F}_r] + D_rY_{\varepsilon,t}E[D_{\theta}D_rY_{\varepsilon,t}|\mathcal{F}_r]\right)dr\right)^2d\theta.
\end{align*}
Using the H\"{o}lder's inequality, we obtain
\begin{align*}
	E\|D\Theta_{\tilde{Y}_{\varepsilon,t}}\|_{L^2[0, T]}^4&\le \frac{t}{\varepsilon^8}\int_0^{t}E\left|\int_{0}^t\left(D_{\theta} D_rY_{\varepsilon,t} E[D_rY_{\varepsilon,t}|\mathcal{F}_r] + D_rY_{\varepsilon,t}E[D_{\theta}D_rY_{\varepsilon,t}|\mathcal{F}_r]\right)dr\right|^4d\theta
	\\
	&\le \frac{t^4}{\varepsilon^8}\int_0^{t}\int_{0}^tE\left|D_{\theta} D_rY_{\varepsilon,t} E[D_rY_{\varepsilon,t}|\mathcal{F}_r] + D_rY_{\varepsilon,t}E[D_{\theta}D_rY_{\varepsilon,t}|\mathcal{F}_r]\right|^4drd\theta\\
	&\le \frac{8t^4}{\varepsilon^8}\int_0^{t}\int_{0}^t \left(E\left|D_{\theta} D_rY_{\varepsilon,t} E[D_rY_{\varepsilon,t}|\mathcal{F}_r]\right|^4+ E\left|D_rY_{\varepsilon,t}E[D_{\theta}D_rY_{\varepsilon,t}|\mathcal{F}_r]\right|^4\right)drd\theta\\
&\le \frac{16t^4}{\varepsilon^8}\int_0^{t}\int_{0}^t \sqrt{E|D_{\theta} D_rY_{\varepsilon,t}|^8 E|D_rY_{\varepsilon,t}|^8}drd\theta.
\end{align*}
This, combined with (\ref{df}) and (\ref{ddf}), gives us the estimate (\ref{Dgamma}). The proof of the proposition is complete.
\end{proof}

\noindent{\it Proof of Theorem \ref{theorem2}.}  Set $\tilde{u}_r = E[D_r\tilde{Y}_{\varepsilon,t}|\mathcal{F}_r]$ for $0\leq r\leq t\leq T.$ We have
\begin{align*}
	\| \tilde{u}\|^8_{L^2[0,T]} = \left[\int_0^T \big|E[D_r\tilde{Y}_{\varepsilon,t}|\mathcal{F}_r]\big|^2dr\right]^4= \frac{1}{\varepsilon^8}\left[\int_0^t \big|E[D_rY_{\varepsilon,t}|\mathcal{F}_r]\big|^2dr\right]^4.
\end{align*}
Using H\"{o}lder's inequality and the estimate (\ref{df}) we deduce
\begin{align}
	E\| \tilde{u}\|_{L^2[0,T]}^8&\le  \frac{t^3}{\varepsilon^8} \int_0^tE\big |E[D_rY_{\varepsilon,t}|\mathcal{F}_r]\big|^{8}dr\notag\\
	&\le  \frac{t^3}{\varepsilon^8}\int_0^tE[|D_rY_{\varepsilon,t}|^8]dr \le Ct^{12}.\label{xDgamma}
\end{align}
We now apply Theorem \ref{lm2} to $F=\tilde{Y}_{\varepsilon,t}$ and $N=Z_t$ to get
\begin{align*}
	I(\tilde{Y}_{\varepsilon,t}\|N_t)\leq c\left(\frac{1}{\gamma_t^4}(E[\tilde{Y}_{\varepsilon,t}])^2+A_{\tilde{Y}_{\varepsilon,t}}|{\rm Var}(\tilde{Y}_{\varepsilon,t})-\gamma_t^2|^2+C_{\tilde{Y}_{\varepsilon,t}}\left(E\|D\Theta_{\tilde{Y}_{\varepsilon,t}}\|_{L^2[0, T]}^4\right)^{1/2}\right),
\end{align*}
where $c$ is an absolute constant and
$$
	A_{\tilde{Y}_{\varepsilon,t}}:= \frac{1}{\gamma_t^4}\left(E\|\tilde{u}\|^{8} _{L^2[0,T]}E|\Theta_{\tilde{Y}_{\varepsilon,t}}|^{-8}\right)^{1/4},
	C_{\tilde{Y}_{\varepsilon,t}}:=A_{\tilde{Y}_{\varepsilon,t}}+\left(E\|\tilde{u}\|_{L^2[0,T]}^8E|\Theta_{\tilde{Y}_{\varepsilon,t}}|^{-16}\right)^{1/4}.
$$
It follows from the estimates (\ref{momenam}) and (\ref{xDgamma}) that
\begin{align*}
&A_{\tilde{X}_{\varepsilon,t}}\le  \frac{Ct^3}{\gamma_t^4}\left(E\bigg|\int_0^t(t-r)^2\sigma^2(r,X_{\varepsilon,r})dr\bigg|^{-p_0}\right)^{\frac{2}{p_0}},\\
&C_{\tilde{X}_{\varepsilon,t}}\le \frac{Ct^3}{\gamma_t^4}\left(E\bigg|\int_0^t(t-r)^2\sigma^2(r,X_{\varepsilon,r})dr\bigg|^{-p_0}\right)^{\frac{2}{p_0}}
+Ct^3\left(E\bigg|\int_0^t(t-r)^2\sigma^2(r,X_{\varepsilon,r})dr\bigg|^{-p_0}\right)^{\frac{4}{p_0}}.
\end{align*}
Furthermore, thanks to Propositions \ref{pr8} and \ref{rsprx}, we have
\begin{align*}
&|E[\tilde{Y}_{\varepsilon,t}]|^2\leq Ct^4\varepsilon^2,\\
&|{\rm Var}(\tilde{Y}_{\varepsilon,t})-\beta_t^2|^2\leq Ct^7\varepsilon^2,\\
&\left(E\|D\Theta_{\tilde{Y}_{\varepsilon,t}}\|^4\right)^{1/2}\leq C t^{7}\varepsilon^2.
\end{align*}
Combining the above computations gives us
\begin{multline*}
I(\tilde{Y}_{\varepsilon,t}\|Z_t)\leq C\bigg(\frac{t^4}{\gamma_t^4}+\frac{t^{10}}{\gamma_t^4}\bigg(E\bigg|\int_0^t(t-r)^2\sigma^2(r,X_{\varepsilon,r})dr\bigg|^{-p_0}\bigg)^{\frac{2}{p_0}}\\+
t^{10}\bigg(E\bigg|\int_0^t(t-r)^2\sigma^2(r,X_{\varepsilon,r})dr\bigg|^{-p_0}\bigg)^{\frac{4}{p_0}}\bigg)\varepsilon^2.
\end{multline*}
So the proof of Theorem \ref{theorem2} is complete.\hfill$\square$

\section*{Appendix. Negative moment of Volterra functionals}

\noindent{\bf Lemma A.} {\it Let $(Y_t)_{t\in[0,T]}$ be a stochastic process such that $Y_0$ is deterministic and $E|Y_t-Y_0|^p\leq Ct^{p\alpha}$ for some $C>0,\alpha>0$ and for all $t\in [0,T],p>1.$ Let $h:[0,T]\times \mathbb{R}\to \mathbb{R}$ and $k:[0,T]^2\to \mathbb{R}_+$ be continuous functions such that $h(0,X_0)\neq 0,$ $k(t,0)\neq 0$ for all $t\in(0,T]$ and
\begin{align}
&|h(t,x)-h(s,y)|\leq L(|t-s|^{\delta_1}+|x-y|^{\delta_2})\,\,\,\forall x,y\in \mathbb{R},s,t\in[0,T]\\\label{7gey3}
&|k(t,s)-k(t,0)|\leq L|t-s|^{\delta_3}\,\,\,\forall s,t\in[0,T]
\end{align}
for some positive constants $L,\delta_1,\delta_2$ and $\delta_3.$ Then, for  $\delta_0:=\min\{\delta_1,\alpha\delta_2,\delta_3\}$  and for all $p_0>0,$ we have
\begin{equation}\label{7gy3}
E\left[\bigg(\int_0^tk(t,r)h^2(r,Y_r)dr\bigg)^{-p_0}\right]\leq C\bigg(\frac{1}{tk(t,0)}\bigg)^{p_0}\left(1+ \bigg(\frac{t^{\delta_0}}{k(t,0)}\bigg)^{p}\right)
\end{equation}
for all $t\in(0,T]$ and $p>\frac{p_0}{\delta_0},$ where $C$ is a positive constant not depending on $t.$}

\begin{proof}Fixed $t\in(0,T].$ Put $y_0:=\frac{2}{th^2(0,Y_0)}.$ Then, for any $y\geq y_0,$ we have $\eta_y:=\frac{2}{tyh^2(0,Y_0)}\in (0,1).$ Hence,
\begin{align*}
\frac{1}{k(t,0)}\int_0^tk(t,r)&h^2(r,Y_r)dr\geq \frac{1}{k(t,0)}\int_0^{\eta_yt}k(t,r)h^2(r,Y_r)dr\\
&= h^2(0,Y_0)\eta_yt+\frac{1}{k(t,0)}\int_0^{\eta_yt}(k(t,r)h^2(r,Y_r)-k(t,0)h^2(0,Y_0))dr\\
&\geq \frac{2}{y}-\frac{1}{k(t,0)}\int_0^{\eta_yt}|k(t,r)h^2(r,Y_r)-k(t,0)h^2(0,Y_0)|dr.
\end{align*}
Consequently, by Markov inequality, we obtain
\begin{align*}
P\bigg(\frac{1}{k(t,0)}&\int_0^tk(t,r)h^2(r,Y_r)dr\leq \frac{1}{y}\bigg)\\
&\leq P\left(\frac{1}{k(t,0)}\int_0^{\eta_yt}|k(t,r)h^2(r,Y_r)-k(t,0)h^2(0,Y_0)|dr\geq \frac{1}{y}\right)\\
&\leq \frac{y^p}{k^p(t,0)}E\left|\int_0^{\eta_yt}|k(t,r)h^2(r,Y_r)-k(t,0)h^2(0,Y_0)|dr\right|^p\,\,\,\forall\,p>1.
\end{align*}
We now use H\"older's inequality to get
\begin{align*}
P&\bigg(\frac{1}{k(t,0)}\int_0^tk(t,r)h^2(r,Y_r)dr\leq \frac{1}{y}\bigg)\\
&\leq \frac{y^p(\eta_yt)^{p-1}}{k^p(t,0)}\int_0^{\eta_yt}E|k(t,r)h^2(r,Y_r)-k(t,0)h^2(0,Y_0)|^pdr\\
&\leq \frac{2^{p-1}y^p(\eta_yt)^{p-1}}{k^p(t,0)}\int_0^{\eta_yt}\left(k^p(t,r)E|h^2(r,Y_r)-h^2(0,Y_0)|^p+h^{2p}(0,Y_0)|k(t,r)-k(t,0)|^p\right)dr\\
&\leq \frac{Cy^p(\eta_yt)^{p-1}}{k^p(t,0)}\int_0^{\eta_yt}\left(E|h^2(r,Y_r)-h^2(0,Y_0)|^p+r^{p\delta_3}\right)dr\,\,\,\forall\,p>1,
\end{align*}
where $C$ is a positive constant not depending on $t.$ We observe that
\begin{align*}
E|h(r,Y_r)-h(0,Y_0)|^{2p}&\leq 2^{2p-1}L^{2p}(r^{2p\delta_1}+E|Y_r-Y_0|^{2p\delta_2})\\
&\leq C(r^{2p\delta_1}+r^{2p\alpha\delta_2})
\end{align*}
for all $r\in[0,T].$  We also have
$$E|h(r,Y_r)+h(0,Y_0)|^{2p}\leq E|h(r,Y_r)-h(0,Y_0)|^{2p}+|2h(0,Y_0)|^{2p}\leq C$$
for all $r\in[0,T].$ So it holds that
\begin{align*}
E|h^2(r,Y_r)-h^2(0,Y_0)|^p&\leq \sqrt{E|h(r,Y_r)+h(0,Y_0)|^{2p}E|h(r,Y_r)-h(0,Y_0)|^{2p}}\\
&\leq C(r^{2p\delta_1}+r^{2p\alpha\delta_2}),\,\,r\in[0,T].
\end{align*}
Combining the above estimates yields
\begin{align*}
P\left(\frac{1}{k(t,0)}\int_0^tk(t,r)h^2(r,Y_r)dr\leq \frac{1}{y}\right)\leq \frac{Cy^p(\eta_yt)^{(1+\delta_0)p}}{k^p(t,0)}\leq \frac{Cy^{-p\delta_0}}{k^p(t,0)}\,\,\,\forall\,p>1,
\end{align*}
Then, for all $p>\frac{p_0}{\delta_0},$ we obtain
\begin{align*}
E\bigg[\bigg(\int_0^tk(t,r)&h^2(r,Y_r)dr\bigg)^{-p_0}\bigg]=\frac{1}{k^{p_0}(t,0)}E\left[\bigg(\frac{1}{k(t,0)}\int_0^tk(t,r)h^2(r,Y_r)dr\bigg)^{-p_0}\right]\\
&=\frac{1}{k^{p_0}(t,0)}\int_0^\infty p_0y^{p_0-1}P\left(\frac{1}{k(t,0)}\int_0^tk(t,r)h^2(r,Y_r)dr\leq \frac{1}{y}\right)dy\\
&\leq \frac{1}{k^{p_0}(t,0)}\left(\int_0^{y_0} p_0y^{p_0-1}dy+ \int_{y_0}^\infty p_0y^{p_0-1}P\left(\int_0^t\frac{1}{k(t,0)}h^2(r,Y_r)dr\leq \frac{1}{y}\right)dy\right)\\
&\leq \frac{1}{k^{p_0}(t,0)}\left(y_0^{p_0}+ \frac{C}{k^p(t,0)}\int_{y_0}^\infty p_0y^{p_0-1}y^{-p\delta_0}dy\right)\,\,\,\forall\,p>1.
\end{align*}
We now choose to use $p=1+\frac{p_0}{\delta_0}.$ This choice gives us the following
$$E\left[\bigg(\int_0^tk(t,r)h^2(r,Y_r)dr\bigg)^{-p_0}\right]\leq C\bigg(\frac{1}{tk(t,0)}\bigg)^{p_0}\left(1+ \bigg(\frac{t^{\delta_0}}{k(t,0)}\bigg)^{p}\right).$$
So we obtain the desired estimate (\ref{7gy3}) because $y_0:=\frac{2}{th^2(0,Y_0)}.$
\end{proof}

\noindent {\bf Acknowledgments.}  The authors would like to thank the anonymous referees for valuable comments which led to the improvement of this work.

\noindent {\bf Declaration of Competing Interest.} The authors declare that they have no known competing financial interests or personal relationships that could have appeared to influence the work reported in this paper.

\end{document}